\documentclass[11pt, a4paper]{amsart}

\usepackage{amsmath,amssymb,amscd,amsfonts}
\usepackage[textwidth=360pt,textheight=615pt]{geometry}
\numberwithin{equation}{section}

\newtheorem{thm}{Theorem}[section]
\newtheorem{lem}[thm]{Lemma}
\newtheorem{rem}[thm]{Remark}
\newtheorem{prop}[thm]{Proposition}

\newtheorem{claim}[thm]{Claim}

\newtheorem{ques}[thm]{Question}
\newtheorem{conj}[thm]{Conjecture}

\newcommand\ct{{\text{\rm ct}}}

\newcommand\cE{{\mathcal{E}}}

\newcommand\cH{{\mathcal{H}}}

\newcommand\cX{{\mathcal{X}}}
\newcommand\cY{{\mathcal{Y}}}

\newcommand\bZ{{\mathbb Z}}
\newcommand\bC{{\mathbb C}}
\newcommand\bQ{{\mathbb Q}}

\newcommand\bP{{\mathbb P}}
\begin{document}
\title{On threefold canonical thresholds}
\author{Jheng-Jie Chen}
\address{\rm Department of Mathematics, National Central University, Taoyuan City, 320, Taiwan}
\email{jhengjie@math.ncu.edu.tw}

\maketitle

\begin{abstract}
We show that the set of threefold canonical thresholds satisfies the ascending chain condition.
Moreover, we derive that threefold canonical thresholds in the interval $ (\frac{1}{2}, 1)$ consists of $ \{ \frac{1}{2}+\frac{1}{n}\}_{n \ge 3} \cup \{ \frac{4}{5}\}$.
\end{abstract}

\section{introduction}
In higher dimensional birational geometry, it is a very natural and important question to measure singularities of a given variety $X$ or more generally, to measure singularities of a given pair $(X, S)$ which consists of a variety and an effective divisor $S$ in $X$. For example, in minimal model program, one hopes to find a good birational model by a sequence of divisorial contractions and flips and also one hopes to understand the birational relations between models.
The termination of three-dimensional terminal flips can be seen by introducing a measurement of complexity of singularities called ``difficulty" (see, for instance \cite{Sho85, KMM87}). Since difficulty is a non-negative integer and it is strictly decreasing after a flip, hence it follows that termination of threefold flips.

Another example is the so-called Sarkisov Program, which try to link two birational models such that each one is a Mori fiber space.
In \cite{Corti}, Corti showed the existence of threefold Sarkisov program which connects two birational Mori fiber spaces by finitely many {\it Sarkisov links}. The key measurement is the Sarkisov degree $(\mu, c, e)$, where canonical threshold $c=\ct(X, \cH)$ plays the more subtle and crucial role. Indeed, as noted in \cite[p233-234]{Corti} if the set of threefold canonical thresholds satisfies the ascending chain condition (ACC), then it follows almost immediately that birational Mori fiber spaces are connected by finitely many Sarkisov links. Hence it is natural and very interesting to consider the following conjecture for canonical thresholds, which is analogous to that of log canonical thresholds and minimal log discrepancies (see \cite{Kol92}, \cite{Kol97}, \cite{MP04}, \cite{Prok08} and \cite{Stepanov} for example).

\begin{conj}\label{ACCcan} The set
$$\mathcal{T}_{n}^{\textup{can}} \colon = \{\ct(X, S) | \dim X = n, S \text{ is integral and effective}\}$$
satisfies the ascending chain condition.
\end{conj}

Notice that Hacon, M$^{\textup{c}}$kernan and Xu showed that the set of log canonical thresholds in any dimension satisfies the ACC in \cite{HMX14}. In other words, they proved the ACC for the set of $0$-log canonical threshold. However, for positive real number $\epsilon$, the ACC conjecture on $\epsilon$-log canonical thresholds remains open, even in dimension three. Note that Conjecture \ref{ACCcan} is the same as the ACC conjecture on $1$-log canonical thresholds when there is no boundary divisor and $S$ is an integral and effective divisor.

The purpose of this article is to show that Conjecture \ref{ACCcan} holds in dimension three.

\begin{thm}\label{3CTACC} The set $\mathcal{T}_{3}^{\textup{can}}$ satisfies the ascending chain condition.
\end{thm}

Furthermore, with more detailed studies, we show that the set $\mathcal{T}_{3}^{\textup{can}}$ is quite sparse in the interval $(\frac{1}{2},1)$. More precisely, we have the following:

\begin{thm}\label{3CT2} Considering the threefold canonical thresholds in the interval $(\frac{1}{2},1)$, we have
 $$\mathcal{T}_{3}^{\textup{can}} \cap (\frac{1}{2}, 1)=\{\frac{1}{2}+\frac{1}{n} \}_{n \ge 3} \cup \{\frac{4}{5} \}.$$
\end{thm}

We would like to remark that after the completion of this work, there are some subsequent works along the same direction. For example, the author and Han, Liu and Luo independently obtain that the accumulation points of the set $\mathcal{T}_{3}^{\textup{can}}$ consists of $\{1/k\ |\ k\geq 2 \textup{ is an integer}\}\cup\{0\}$ and also generalize Theorem \ref{3CTACC} to pairs (see \cite[Theorems 1 and 2]{Ch22} and \cite[Theorems 1.7 and 1.8]{HLL22}). Moreover, Han, Liu and Luo show that the ACC holds for minimal log discrepancies of terminal threefolds (see \cite[Theorems 1.1 and 4.1]{HLL22}).

We now briefly explain the idea of the proofs of Theorems \ref{3CTACC} and \ref{3CT2}.
Suppose $(X,S)$ is a pair and let $\pi: \tilde{X} \to X$ be a log resolution of $(X, S)$.
We have $$ K_{\tilde{X}} \sim_{\bQ} \pi^*K_X+ \sum_i {a_i} E_i,$$ and
$$ \pi^*S\sim_{\bQ} S_{\tilde{X}} + \sum_i {m_i} E_i,$$ where $S_{\tilde{X}}$ is the proper transform of $S$ and $\sim_{\bQ}$ denotes the $\bQ$-linear equivalence. The canonical threshold, denoted $\ct(X,S)$ is defined as
$$\ct(X,S)= \sup\{ \lambda | (X, \lambda S) \text{ is canonical}\} = \min_i\{ \frac{a_i}{m_i}\}.$$
It is easy to see that the canonical threshold is independent of resolution.

By blowing up along a curve not contained in singular sets of $S$ and $X$, it is also easy to see that $\ct(X,S) \le 1$ by definition.
Let $q\colon X'\to X$ be a $\bQ$-factorialization of $X$ and $S'$ be the strict transform of $S$ on $X'$. Since $q\colon X'\to X$ is a small morphism (see \cite[Corollary 4.5]{Kaw88}), $\ct (X,S)=\ct(X',S')$. By replacing $X$ with $X'$ (resp. $S$ with $S'$), we may assume $X$ is $\bQ$-factorial in this work.
We say that a proper morphism $\sigma\colon Y\to X$ with connected fibers is a divisorial contraction if $-K_Y$ is  $\sigma$-ample and the exceptional locus of $\sigma$ is a prime divisor where both $X, Y$ are normal varieties with at worst terminal singularities. The following is a well-known fact (see e.g. \cite[Corollary 4.5]{Kaw88} and \cite{Corti}).
\begin{prop}\label{realizect} Given a pair $(X,S)$ of a 
terminal threefold $X$ and an integral and effective Weil divisor $S$.
There exists a divisorial contraction computing $\ct(X,S)$. More precisely, there exists a divisorial contraction
$\sigma: Y \to X$ with $K_Y=\sigma^*K_X + a E$, $\sigma^*S= S_Y + m E$ such that $\ct(X,S)=\frac{a}{m}$ where $S_Y$ denotes the strict transform of $S$ on $Y$.
In this case, we say that $\ct(X,S)$ is computed by  $\sigma$.
\end{prop}

Note that if $\sigma$ is a divisorial contraction to a curve, then $a=1$ and hence $\ct(X,S) \in \aleph \colon = \{ \frac{1}{n} \}_{n \ge 1}$, which is clearly an ACC set.
Suppose now that $\ct(X,S)$ is computed by  $\sigma$ for some divisorial contraction to a point $P \in X$. Let $n \in \mathbb{N}$ be the index of $P \in X$. That is, $n$ is the smallest positive integer such that $nK_X$ is Cartier at $P$. 
By abuse of notation, we write $K_{Y}= \sigma^* K_{X} +\frac{a}{n} E$ and $\sigma^* S= \tilde{S} + \frac{m}{n} E$, where $E$ is the exceptional divisor of $\sigma$ and $\tilde{S}$ is the proper transform of $S$. We call $a$ the {\it weighted discrepancy of $\sigma$} and $m$ the {\it weighted multiplicity of $S$ with respect to $\sigma$}. Now $\ct(X,S)=\frac{a}{m}$. It is then sufficient to study various constraints of weighted discrepancy $a$ and weighted multiplicities $m$.

We will use the classification of divisorial contractions to points, due to Kawamata, Kawakita, Hayakawa and some others (cf. \cite{Ka,Haya99,Haya00,Kawakita01,Kawakita02,Kawakita05}). As a consequence,  it is known that weighted discrepancy is bounded by $4$ except some series of weighted blow ups.

\begin{thm}\cite{Kawakita01,Kawakita05}\label{3-folddc}
Let $\sigma: Y \to X \ni P$ be a divisorial contraction to a closed point $P \in X$. If the weighted discrepancy $a > 4$, then one of the following holds:
\begin{enumerate}
  \item $P \in X$ is a smooth point and $\sigma: Y \to X$ is a weighted blow up;
  \item $P \in X$ is a $cA, cA/n, cD$ or $cD/2$ point and $\sigma: Y \to X$ is a weighted blow up satisfying some extra conditions.
\end{enumerate}
\end{thm}

Note that if $\sigma$ is a divisorial contraction to a point with weighted discrepancy $\le 4$, then  $\ct(X,S) \in \aleph_4:= \{ \frac{a}{m} \}_{a \le 4, m \ge 1}$, which is clearly an ACC set. Therefore, it is sufficient to consider the set of canonical thresholds that are computed by divisorial contractions to points with weighted discrepancies $\ge 5$.
By the above classification, these divisorial contractions can be realized as weighted blow ups, hence we will work on canonical thresholds computed by weighted blow ups.

According to the types of the center $P \in X$, we introduce
 \begin{equation*}
 \mathcal{T}_{3, *}^{\textup{can}}=\left\{\ct(X,S) \Bigg| \begin{array}{l}  \ct(X,S)  \text{ is computed by } \sigma: Y \to X \ni P\\
                                      \text{with weighted discrepancy }  \ge 5 \\
                                      \text{ over a point } P \text{ of type } *  \end{array} \right\}, \eqno{\dagger_1}
 \end{equation*}
where the type $*$ could be $sm$, (resp. $cA$, $cA/n$, $cD$, $cD/2$) if $P \in X$ is a smooth point (resp. singular point of type $cA$, $cA/n$, $cD$ or $cD/2$).

The classification of divisorial contractions then implies the following decomposition of sets:
 $$\mathcal{T}_{3}^{\textup{can}} = \aleph_4 \cup \mathcal{T}_{3,sm}^{\textup{can}} \cup \mathcal{T}_{3, cA}^{\textup{can}} \cup \mathcal{T}_{3, cA/n}^{\textup{can}} \cup \mathcal{T}_{3,cD}^{\textup{can}} \cup \mathcal{T}_{3, cD/2}^{\textup{can}}. \eqno{\dagger_2}$$

For each divisorial contraction $\sigma: Y \to X \ni P$ with weighted discrepancy $ \ge 5$, not only there exists explicit description of singularities $P\in X$ (if $P$ lies in the singular locus) but also $\sigma$ is known to be a weighted blow up with certain weight $w$. Let $a,m$ be its weighted discrepancy and weighted multiplicity respectively.  Consider now another weighted blow up $\sigma': Y' \to X \ni P$ with another weight $w'$ such that the exceptional divisor $E'$ is a prime divisor. Let $a',m'$ be its weighted discrepancy and weighted multiplicity respectively. Since $E'$ corresponds to a valuation, then one has $ \frac{a'}{m'} \ge \frac{a}{m}=\ct(X,S)$.
Suppose also that $w' \succeq \mu w$ for some $\mu >0$ (where the notion $\succeq$ is defined in the last paragraph before Lemma \ref{mainlem}), then $m' \ge \mu m$. Roughly speaking, these two inequalities provide estimation of a given $\ct(X,S)$. With careful choices of weight $w'$ and studies of each divisorial contractions to points, we are able to conclude that $\mathcal{T}_{3, *}^{\textup{can}}$ satisfies the ACC and its intersection with the interval $(\frac{1}{2},1)$ is contained in $\{\frac{1}{2}+\frac{1}{n} \}_{n \ge 3}$, and hence the main theorems follow.

The article is organized as follows. In section 2, we fix some notations. We study the cases over smooth point, $cA$, $cA/2$, $cD$, $cD/2$ respectively in section 3-7 respectively. The main theorem then follows from the studies of these cases. In the last section 8, we list a few related questions that could be interesting topics for further investigation.

\noindent
{\bf Acknowledgement.}
The author was partially supported by
NCTS and MOST of Taiwan.
He expresses his gratitude to Professor Jungkai Alfred Chen for extensive help and invaluable discussion and suggestions.
He would like to thank Professor Atsushi Ito for the useful information of finiteness properties of the semi-group $\mathbb{Z}^n_{\ge 0}$ (in \cite{Khov}).
He is grateful to the anonymous referees for very helpful comments, corrections and modifications.

\section{Notations and Conventions.}
We always work over complex number $\bC$.

Let $r\ge 2$ and $n$ be positive integers. Let $\zeta$ be a primitive $n$-th root of unity and $G=\langle \zeta \rangle$ be a finite cyclic group of order $n$ acting on $\hat{\mathbb{C}}^r$ by $x_i \mapsto \zeta^{b_i}x_i$ for all $i$ where $\hat{\mathbb{C}}^r$ denotes the completion of $\bC^r$.
The resulting quotient space $\hat{\mathbb{C}}^r/G$ is usually denoted by $\hat{\mathbb{C}}^r/\frac{1}{n}(b_1,...,b_r)$.

On $\cX:=\hat{\mathbb{C}}^r/\frac{1}{n}(b_1,...,b_r)$, one can consider weighted blow ups $\sigma_w: \cY \to \cX$ with admissible weight $w= \frac{1}{n}(k_1,..., k_r)$. The weight $w$ is said to be admissible if for all $i$, $k_i>0$ and $k_i\equiv sb_i$ (mod $n$) for some integer $s$. Indeed, the construction of weighted blowup is transparent in term of toric geometry. Let $\{ e_1,\ldots e_r\}$ be the standard basis of $\mathbb{R}^r$, and $\sigma$ be the non-negative cone generated by these vectors. Let $\overline{N}$ be the free abelian group generated by $\{e_1,\ldots e_r\}$ and $\frac{1}{n}(b_1,...,b_r)$  and $\overline{M}=\textrm{Hom}_\mathbb{Z}(\overline{N}, \mathbb{Z})$ be its dual. Then $\cX=\textrm{Spec}\ \mathbb{C}[\sigma^\vee \cap \overline{M}]$. So $w$ is admissible means that $w \in \overline{M}$. The weighted blowup $\sigma_w: \cY \to \cX$ is the toric morphism obtained by subdivide the cone $\sigma$ along the ray generated by $w$. For more details, please see  \cite[Section 10]{KM92} or \cite[\textsection 3]{Haya99}.

Note that $\cY$ is covered by affine open subsets $U_1,...,U_r$ where each
\[U_i:=(\overline{x}_1,\cdots,\overline{x}_r)/\frac{1}{k_i}(-k_1,..,\stackrel{i\tiny\textup{-th}\normalsize}{n},...,-k_r),\] where $\sigma_w$ is described as
\[ \sigma_w|_{U_i}\colon U_i\ni (\overline{x}_1,\cdots,\overline{x}_r)\mapsto (\overline{x}_1\overline{x}_i^{k_1/n},...,\stackrel{i\tiny\textup{-th}\normalsize}{\overline{x}_i^{k_i/n}},...,\overline{x}_r\overline{x}_i^{k_r/n})\in \cX, \]
and the exceptional divisor $\cE$ of $\sigma_w: \cY \to \cX$  is isomorphic to the weighted projective space $\mathbb{P}(k_1, k_2,...,,k_r)$.

For any monomial $\frak{m}=x_1^{i_1}x_2^{i_2}\cdots x_r^{i_r}$, we define
the weight of $\frak{m}$ to be $$w(\frak{m})=\sum_{j=1}^r \frac{i_j k_j}{n}.$$
Let $h=\sum a_{i_{1}i_{2}\cdots i_{r}} x_1^{i_1}x_2^{i_2}\cdots x_r^{i_r}$ be a non-zero $G$-semi-invariant formal power series (or polynomial) with variables $x_1, x_2,..., x_r$. That is, there exists an integer $s$ such that $s\equiv \sum_{j=1}^r i_j k_j$ (mod $n$) for all $(i_1,i_2,...,i_r)$ with coefficient $a_{i_{1}i_2 \cdots i_{r}}\neq 0$.
We say that the monomial $\frak{m}=x_1^{i_1}x_2^{i_2}\cdots x_r^{i_r}$ appears in $h$, denoted by $\frak{m}\in h$, if the coefficient $a_{i_{1}i_2\cdots i_{r}}\neq 0$.  Then we define
$$w(h):=\min\{\ w(\frak{m}) \mid  \frak{m} \in h\}.$$ 
Let $h^w$ denote the homogeneous part with minimal weight (with respect to $w$). Then we can write $$h=h^w+ \text{terms of higher weights}.$$

The extended Newton diagram $\Gamma^+(h)$ 
is defined to be the convex hull in $\mathbb{R}^r$ of the set $$\{(i_{1}, i_2,..., i_{r})+\mathbb{R}^r_{\geq 0}\ |\ a_{i_{1} i_2 \cdots i_{r}}\neq 0\ \}.$$

It is known that a threefold terminal singularity $P \in X$ is an isolated singularity which is a $cDV$ cyclic quotient. That is,  $P\in X$ is analytically locally isomorphic to  $ (\varphi=0) \subset \hat{\mathbb{C}}^4/\frac{1}{n}(b_1,b_2,b_3,b_4)$, where $\varphi$ is a semi-invariant having isolated compound Du Val singularity at the origin and $\hat{\mathbb{C}}^4$ denotes the completion of $\mathbb{C}^4$. Detailed classification can be found in \cite{Mori85,YPG} and \cite{KM92}, for example. Recall that $n$ is the index of $P\in X$ which is the smallest positive integer $n$ with $nK_X$ Cartier.

Let $S$ be a Weil divisor on $X$, which is affine and assumed to have $\bQ$-factorial singularity. It follows from \cite[Corollary 5.2]{Kaw88} that there exists an integer $0<e_P\le n$ such that $S\sim e_pK_X$ at $P$ where $\sim$ denotes the linear equivalence (via canonical cover). As the local equation of $K_X$ is semi-invariant (see \cite[p362]{YPG} or \cite[Proposition-Definition 4-5-1]{Mat02}), $S$ is then analytically locally given by $(f=0)\subset \hat{\mathbb{C}}^4/\frac{1}{n}(b_1,b_2,b_3,b_4)$ for some semi-invariant formal power series $f$.

Fix the analytically local embedding of $X$ into $\cX$ and let $Y$ be the proper transform of $X$ in $\cY$, then by abuse of the notation, we also call the induced map $\sigma_w: Y \to X$ the weighted blow up with weight $w$. Let $\cE$ be the exceptional divisor of $\sigma_w: \cY \to \cX$, which is isomorphic to $\bP(k_1,\ldots, k_r)$. Then the exceptional set of the induced weighted blowup $\sigma_w: Y \to X$  is given by $E:=Y \cap \cE$ which is isomorphic to $(\varphi^w=0) \subset \bP(k_1,\ldots, k_r)$. Suppose that $E$ is a prime divisor, or equivalently the ideal $(\varphi^w)$ is prime in the graded ring $\bC[x_1,...,x_r]$ where the weight of $x_i$ is $a_i$ for every $i=1,...,r$, then $E$ gives rise to a valuation on $Y$.

We consider now a pair $(X,S)$ analytically locally embedded into $\cX$.  
Suppose that the exceptional set $E$ of the weighted blow up $\sigma_w: Y \to X$ is a prime divisor.
Then it is straightforward to check the following (cf. \cite[the proof of Theorem 2]{Mori85}, \cite[p373]{YPG} or \cite[\textsection 3]{Haya99}) \begin{align*}
 K_Y&= \sigma_w^* K_X + \frac{a}{n}E, \quad S_Y=\sigma_w^*S - \frac{m}{n}E,
 \end{align*} where

 \begin{equation} \frac{a}{n}= (\sum_{i=1}^4 \frac{k_i}{n}) -w(\varphi)-1, \quad \frac{m}{n}=w(f).  \end{equation}
The number $a$ (resp. $m$) is called the {\it weighted discrepancy (resp. weighted multiplicity) with respect to $w$}.

There are several cases that we need to consider analytically local embedding $(X \ni P)$ into $\hat{\bC}^5$ defined by  $(\varphi_{1}=\varphi_{2}=0) \subset \hat{\mathbb{C}}^5/\frac{1}{n}(b_1,b_2,b_3,b_4,b_5)$ (resp.  $(X \ni P) \simeq (\hat{\bC}^3\ni o)$).
In this situation, the previous discussion can be carried over naturally. In particular, the exceptional divisor $E \simeq (\varphi_1^w=\varphi_2^w=0)\subset \mathbb{P}(k_1,k_2,k_3,k_4,k_5)$  (resp.  $E \simeq \bP(k_1,k_2,k_3)$). Moreover, the weighted discrepancy $a$ and the weighted multiplicity $m$ are computed by \begin{equation}
\frac{a}{n}=\sum_{j=1}^5 \frac{k_j}{n} - w(\varphi_1)-w(\varphi_2)-1 \textup{ (resp. }\frac{a}{n}=\sum_{j=1}^3 \frac{k_j}{n}-1),\  \frac{m}{n}=w(f). \end{equation}

Recall that the canonical threshold $\ct(X,S)$ of $(X,S)$ is computed by a divisorial contraction $\sigma\colon Y\to X$ by Proposition \ref{realizect}. As we explained in the previous section, it is enough to consider the set of canonical threshold with weighted multiplicities $a \ge 5$, which are computed by weighted blowups by Theorem \ref{3-folddc}.   Therefore, we are naturally led to study various weighted blow ups.
We are particularly interested in comparing the weights computing canonical threshold and its ``approximations''.

We now compare two different weights when an analytically local embedding of $X$ into $\cX$ is fixed. Given two weights $w=\frac{1}{n}(k_1,...,k_r)$ and
$w'=\frac{1}{n}(k'_1,..,k'_r)$ where $r\ge 2$ is a positive integer, we say that
\begin{equation*} w' \succeq \mu w \text{ if } k'_i \ge \mu k_i \text{ for all } i. \end{equation*}

\begin{lem}\label{mainlem}
Given a pair $(X,S)$ such that $S$ is an effective integral divisor and the canonical threshold is computed by weighted blow up with weight $w$ under an analytically local embedding $e:X\to \cX$. Let $a, m$ be its weighted discrepancy and weighted multiplicities respectively so that $\ct(X,S)=\frac{a}{m}$.

Suppose now that there is another weight $w'$ such that the exceptional set of the weighted blow up (under the same embedding $e:X\to \cX$) is a prime divisor.
Let $a', m'$ be its weighted discrepancy and weighted multiplicities respectively. Then
\begin{equation}\label{eq1}
m' \le \lfloor \frac{a'}{a}m \rfloor \le \frac{a'}{a}m
\end{equation}

On the other hand, suppose furthermore that $w' \succeq \mu w$ for some real number $\mu$, then
\begin{equation}\label{eq2}
m' \ge \lceil \mu m \rceil \ge \mu m
\end{equation}
\end{lem}

\begin{proof} Let $\sigma_{w'}:Y' \to X$ be the weighted blowup with weight $w'$. If its exceptional set is a prime divisor, then it defines a valuation. By the definition of canonical threshold, we have $\frac{a}{m} =\ct(X,S) \le \frac{a'}{m'}$.
Therefore the first inequality holds since $m'$ is an integer.
The second inequality is straightforward.
\end{proof}

These two very elementary inequalities play pivotal roles in our arguments.
Therefore, we need to check the irreducibility of exceptional divisors, which is equivalent to the irreducibility of defining equations.

The following lemma is useful.

\begin{lem}\label{irred}
Let $f$  be a  polynomial in a graded polynomial ring $\bC[x, y, z ,u]$.
If $f$ is of one of the following form (possibly after change of coordinates), then $f$ is irreducible.
\begin{enumerate}
\item $f=xy^p+ yg_1(y,z,u)+g_0(z,u)$ for some integer $p \ge 1$ with $g_0(z,u) \ne 0$.

\item $f=x(z g_1(y,z,u)+g_0(y,u))+z^p$ for some integer $p \ge 1$ with $g_0(y,u) \ne 0$.
\end{enumerate}
\end{lem}

\begin{proof} Suppose on the contrary that $f$ is reducible. Since $f$ is linear in $x$, we may write \[ f=(h_1 x + h_0)q=qh_1x+qh_0\] for some $h_1, h_0, q \in \bC[y,z,u]$ and $q$ is non-constant.

In case (1), $y^p=qh_1$ and $yg_1(y,z,u)+g_0(z,u)=qh_0$. We may assume that $q=y^k$ for some $1 \le k \le p$ and hence $y|g_0(z,u)$, which is a contradiction.

In case (2), $z^p=qh_0$ and $zg_1(y,z,u)+g_0(y,u)=qh_1$. We may assume that $q=z^k$ for some $1 \le k \le p$ and hence $z|g_0(z,u)$, which is a contradiction.
\end{proof}

\begin{thm}\cite[Theorem 5]{Ka}\label{Kawbu}
Let $(Y,Q)$ be a germ of terminal cyclic quotient singularity $\hat{\mathbb{C}}^3/\frac{1}{n}(s,-s,1)$ with $s<n$ and $n, s$ are coprime positive integers. If $\tau\colon Z\to Y$ is a divisorial contraction with exceptional divisor $E'$ such that 
$Q\in \tau(E')$, then $\tau$ is uniquely determined by the weighted blow up with weight $\frac{1}{n}(s,n-s,1)$. In particular, $\tau(E')=Q$.
\end{thm}

The divisorial contraction $\tau:Z\to Y$ in Theorem \ref{Kawbu} is called Kawamata blow up.

\section{canonical thresholds in $\mathcal{T}_{3, sm}^{\textup{can}}$ }

First recall the following two results of Stepanov:
\begin{prop}\cite[Theorem 1.7]{Stepanov} \label{Stepanov}
The set $\mathcal{T}_{3, sm}^{\textup{can}}$ satisfies the ascending chain condition.
\end{prop}

\begin{prop}\cite[Theorem 3.6]{Stepanov} \label{Brieskorn}
Let $S \subset \hat{\mathbb{C}}^3$ be a Brieskorn singularity of the form $x^a+y^b+z^c=0$ with $2 \le a \le b \le c$. Then $ \ct(\hat{\mathbb{C}}^3, S) = \frac{1}{a}+\frac{1}{b}$ if $ c \ge \textup{lcm}(a, b)$.
\end{prop}

Suppose that $\ct(X,S) \in \mathcal{T}_{3,sm}^{\textup{can}}$. Kawakita \cite{Kawakita01} shows that the canonical threshold $\ct(X,S)$ is computed by $\sigma_w: Y \to X \ni P$ which is a weighted blow up over a smooth point $P$ of weight $w=(1,\alpha,\beta)$, where $\alpha, \beta$ are relatively prime integers and $1\le \alpha<\beta$. The weighted discrepancy is $\alpha+\beta$.
Let $m$ be the weighted multiplicity of $(X,S)$ with respect to $\sigma_w$, then we have $\ct(X,S)=\frac{\alpha+\beta}{m}$.

\begin{prop}\label{sm_ineq}
Keep the notations as above and suppose furthermore that $\alpha+\beta\nmid m$ and $\alpha>1$. The following statements hold.
\begin{enumerate}
\item We have $m \ge \alpha\beta$ and $\ct(X,S)=\frac{\alpha+\beta}{m}\leq \frac{1}{\alpha}+\frac{1}{\beta}$.
\item If $m=\alpha\beta$, then ${\ct(X,S)}=\frac{1}{\alpha}+\frac{1}{\beta}$.
\item  $m \ne \alpha\beta+1$.
\end{enumerate}
\end{prop}

\begin{proof}
Note that there are integers $s,t$ with $0<s\le \alpha$ and $0<t\le \beta$ and
$$ \alpha t=\beta s+1.$$
Let $\bar{s}:=\alpha-s$ and $\bar{t}:=\beta-t$. Then we also have
$$ \alpha \bar{t}=\beta \bar{s}-1.$$
Recall that $Y$ is covered by three affine open subsets $U_1, U_2, U_3$ where $U_1\simeq \hat{\mathbb{C}}^3$ and
 \begin{align*}
U_2&\simeq \hat{\mathbb{C}}^3/\frac{1}{\alpha}(-1,1, -\beta)\simeq \hat{\mathbb{C}}^3/\frac{1}{\alpha}(-s,s, 1) \textup{ and } \\
U_3&\simeq \hat{\mathbb{C}}^3/\frac{1}{\beta}(-1,-\alpha, 1)\simeq \hat{\mathbb{C}}^3/\frac{1}{\beta}(t,1, -t).
\end{align*}
Define $$v_2=\frac{1}{\alpha}(\bar{s}, s, 1),\ v_3:=\frac{1}{\beta}(t,1, \bar{t}),\ w_2=\frac{s}{\alpha}w+\frac{1}{\alpha}(\bar{s},0,1), w_3= \frac{\bar{t}}{\beta}w+\frac{1}{\beta}(t,1,0).$$
Then $w_2=(1,s,t)$, $w_3=(1, \bar{s}, \bar{t})$. Let $m_2,m_3$ be their weighted multiplicities respectively.
Recall that the further weighted blow up at the origin of $U_2$ (resp. $U_3$) with weight $v_2$ (resp. $v_3$) is Kawamata blow up. 
Note that the weighted discrepancies of $\sigma_{w_i}$ are $s+t$ and $\bar{s}+\bar{t}$ respectively. 
Notice also that
$w_2 \succeq \frac{s}{\alpha} w$ and $w_3 \succeq  \frac{\bar{t}}{\beta} w$. Now it follows from the inequalities in Lemma \ref{mainlem} that $$\lfloor  \frac{s+t}{\alpha+\beta} m \rfloor + \lfloor  \frac{\bar{s}+\bar{t}}{\alpha+\beta} m \rfloor \ge m_2+m_3 \ge   \lceil \frac{s}{\alpha} m \rceil + \lceil \frac{\bar{t}}{\beta} m \rceil.$$

First suppose that $m < \alpha\beta$.
Since $$\frac{s}{\alpha}+\frac{\bar{t}}{\beta} = \frac{s}{\alpha} + \frac{\bar{s}}{\alpha} -\frac{1}{\alpha\beta}=1-\frac{1}{\alpha\beta},$$
it follows that
$$ \lceil \frac{s}{\alpha} m \rceil +\lceil \frac{\bar{t}}{\beta} m \rceil \ge  \lceil \frac{s}{\alpha} m+\frac{\bar{t}}{\beta} m \rceil = \lceil m - \frac{m}{\alpha\beta} \rceil =m. $$
On the other hand, since $\frac{s+t}{\alpha+\beta} m + \frac{\bar{s}+\bar{t}}{\alpha+\beta} m =m$,
one has that
$$ \lfloor \frac{s+t}{\alpha+\beta} m \rfloor + \lfloor \frac{\bar{s}+\bar{t}}{\alpha+\beta} m \rfloor= \left\{ \begin{array}{ll} m-1 & \text{ if } \frac{s+t}{\alpha+\beta} m \not \in \bZ; \\ m & \text{ if } \frac{s+t}{\alpha+\beta} m  \in \bZ. \end{array} \right.
$$

Notice that $\gcd(\alpha+\beta, s+t)=1$ are relatively prime. Hence $\frac{s+t}{\alpha+\beta} m  \in \bZ$ if and only if $\alpha+\beta |m$. We thus conclude the statement (1).

It remains to show (3).
If $m=\alpha\beta+1$, one can easily compute that
$$\lceil \frac{s}{\alpha} m \rceil +\lceil \frac{\bar{t}}{\beta} m \rceil =\lceil \frac{s}{\alpha} (\alpha\beta+1) \rceil +\lceil \frac{\bar{t}}{\beta} (\alpha\beta+1) \rceil= \alpha\beta+1.$$
Since $\alpha+\beta\nmid m$, then $\frac{s+t}{\alpha+\beta} m  \not \in \bZ$, which leads to a contradiction.
\end{proof}

\begin{prop}\label{sm2}
We have $\mathcal{T}_{3,sm}^{\textup{can}} \cap (\frac{1}{2},1) = \{\frac{1}{2}+\frac{1}{t}\}_{t \ge 3}$.
\end{prop}

\begin{proof}
For every integer $t>2$,
we consider the Brieskorn singularity $x^2+y^t+z^c=0$ where $c$ is an integer $\ge \textup{lcm}(2, t)$. It has canonical threshold $\frac{1}{2}+\frac{1}{t}$ by Proposition \ref{Brieskorn}.

On the other hand, suppose that $\ct(X,S)\in \mathcal{T}_{3,sm}^{\textup{can}} \cap (\frac{1}{2},1)$ is computed by a weighted blow up with weight $(1,\alpha,\beta)$ with  $1 \le \alpha <\beta $. Then by Proposition \ref{sm_ineq}$, \frac{1}{2} < \ct(X,S) \le \frac{1}{\alpha}+\frac{1}{\beta}$ if $\alpha>1$. It follows that  $\alpha \le 3$.

Suppose first that $\alpha=3$, then $\beta=4,5$. Hence $\ct(X,S)$ could possibly only be $\frac{7}{12}=\frac{1}{2}+\frac{1}{12}$ or $\frac{8}{15}=\frac{1}{2}+\frac{1}{30}$.

Suppose now that $\alpha=2$. Since $\frac{2+\beta}{m}= \ct(X,S) > \frac{1}{2}$, so we have $m=2\beta, 2\beta+2$ or $2\beta+3$ but not $2\beta+1$ by Proposition \ref{sm_ineq}. Note that
$$\frac{\beta+2}{2\beta}=\frac{1}{2}+\frac{1}{\beta}, \quad \frac{\beta+2}{2\beta+2}=\frac{1}{2}+\frac{1}{2\beta+2}, \quad
\frac{\beta+2}{2\beta+3}=\frac{1}{2}+\frac{1}{4\beta+6}.$$

Finally, let us assume that $\alpha=1$. We would like to compare the weight $(1,1,\beta)$ with the weight $w_1:=(1,1,\beta-1)$. 
By the inequalities in Lemma \ref{mainlem}, we have
$$ \lfloor \frac{\beta}{\beta+1} m \rfloor  \ge \lceil \frac{\beta-1}{\beta} m \rceil.\eqno{\dagger_3} $$

Since $ \frac{1}{2} < \frac{\beta+1}{m} < 1$, we have $\lfloor \frac{\beta}{\beta+1} m \rfloor=m-2$. However, if $m <2\beta$, then $\lceil \frac{\beta-1}{\beta} m \rceil = m-1$, which is a contradiction to $\dagger_3$. Therefore, we have $m=2\beta$ or $2\beta+1$. So now
$$\ct(X,S)= \frac{\beta+1}{2\beta}=\frac{1}{2}+\frac{1}{2\beta},\mbox{ or } \ct(X,S)= \frac{\beta+1}{2\beta+1}=\frac{1}{2}+\frac{1}{4\beta+2}.$$
Therefore, we conclude that
$\mathcal{T}_{3,sm}^{\textup{can}} \cap (\frac{1}{2},1) \subseteq \{\frac{1}{2}+\frac{1}{t}\}_{t \ge 3}$. This completes the proof.
\end{proof}

\section{canonical thresholds in $\mathcal{T}_{3, cA}^{\textup{can}}$}
In this section, we consider the canonical thresholds over $cA$ points. In viewing $\dagger_1$ and $\dagger_2$,  it is sufficient to consider divisorial contractions which are weighted blowups with $a \ge 5$. Hence we will assume that $a \ge 5$ in this section.

Let  $\ct(X,S) \in \mathcal{T}_{3, cA}^{\textup{can}}$ be a canonical threshold realized by a divisorial contraction $\sigma: Y \to X$. Theorem 1.2(i) in \cite{Kawakita05} shows that there exists an analytical identification $P\in X \simeq o\in (\varphi=xy+g(z,u)=0)$ in $\hat{\mathbb{C}}^4$ where $o$ denotes the origin of $\hat{\mathbb{C}}^4$ and $\sigma$ is a weighted blow up of weight $w=wt(x,y,z,u)=(r_{1},r_{2},a,1)$ satisfying the following:

\begin{itemize}
\item $w(g(z,u))=r_{1}+r_{2}=ad$ where $r_1,r_2,a,d$ are positive integers;
\item $z^{d}\in g(z,u)$ and hence $w(g(z,u))=w(z^{d})$;
\item $\gcd(r_{1}, a)=\gcd(r_{2}, a)=1$.
\end{itemize}

We may assume $d\ge 2$ since otherwise $P\in X$ is nonsingular and is treated in the previous section.
Suppose that $S$ is defined by the formal power series $f=0$ analytically and locally, then $\ct(X,S)=\frac{a}{m}$, where $m=w(f)$.

\begin{lem} \label{lower}
Keep the notation as above. Suppose that there is another weight $w'=(r'_1, r'_2, a',1)$ satisfying $r'_1+r'_2=a'd$ and $a' \leq a$. Let $\sigma': Y' \to X$ be the weighted blow up with weight $w'$. Then the exceptional set $E'$ of $\sigma'$ is a prime divisor.
\end{lem}

\begin{proof}
It is convenient to consider {\it truncated weight} $v=(a,1)$ and $v'=(a',1)$ on $\{z, u \}$. Note that $w(\varphi)=v(g)=ad$.
Since $v' \succeq \frac{a'}{a} v$, one has
$$v'(g) \ge \frac{a'}{a} v(g) = a'd,$$ and hence $w'(\varphi)=a'd$.

Suppose that $a' \le a$. One has \[ \varphi^{w'}=xy+g^{w'}(z,u),\]
where $g^{w'}(z,u)$ is the $w'$-weighted homogeneous part of $g(z,u)$. As $z^d \in g^{w'}(z,u)$, one has $g^{w'}(z,u) \ne 0$ and hence $\varphi^{w'}$ is irreducible by Lemma \ref{irred}.(1). The exceptional set $E'$ of $\sigma'$, which is a hypersurface in $\bP(r_1', r_2', a',1)$ defined by $\varphi^{w'}$, is thus irreducible.
\end{proof}

By interchanging $x$ and $y$ if necessary, we may assume $r_1\leq r_2$ and hence $r_2 >1$. We will employ the method developed in \cite{CJK15} to produce two weights $w_1, w_2$, complementary to each other in some sense,  approximate the weight computing canonical threshold.
These two weights, together with the inequalities in Lemma \ref{mainlem}, provide very tight constraints to the canonical threshold which allows us to determine them completely in the interval $(\frac{1}{2},1)$.

\begin{prop} \label{cAub}
Suppose that $\ct(X,S)=\frac{a}{m}$ is realized over a $cA$ point $P \in X$. Suppose furthermore that $a\nmid m$ and min$\{r_1,r_2\}=r_1>1$.
Then $m \ge \frac{r_1 r_2}{d}$ and hence $\frac{a}{m} \le \frac{1}{r_1}+ \frac{1}{r_2}$.
\end{prop}
\begin{proof}
Following \cite[Section 4, Case Ib]{CJK15}, there exist positive integers $a_1,a_2,s_1^*,s_2^*$ such that $1+a_1 r_1= s_1^*a$, $1+ a_2 r_2=s_2^*a$ with $a_1+a_2=a$.
Clearly, $\gcd(a,a_1)=1$, $a_1<a, a_2<a, s_2^*<r_2$ and $s_1^*<r_1$.

 We consider \footnote{The idea of \cite{CJK15} is to consider a nicely chosen weighted up over the singular points of high Cartier index $r_i$ on $Y$. Then the so-called 2-ray game will lead to another extraction which is a weighted blow up with different weight. This is how $w_1,w_2$ are chosen. In fact, in the language of toric geometry, $Y$ is obtained by a subdivision of a cone along the vector $w$. Further weighted blow up (which we choose Kawamata blow up in our case) corresponds to further subdivision along a vector. Explicit computation of these vectors explicitly gives $w_1, w_2$. }
  $$ w_2=(r_1-a_2d+s_2^*, r_2-s_2^*, a_1,1);$$
$$w_1=(r_1-s_1^*, r_2-a_1d+s_1^*, a_2, 1).$$

Note that $w_1\succeq \frac{r_1-s_1^*}{r_1} w$ and $w_2\succeq \frac{r_2-s_2^*}{r_2} w$.
Let $m_i=w_i(f)$ for $i=1,2$. By Lemma \ref{lower} and inequalities in Lemma \ref{mainlem}, we see that
$$ \lfloor \frac{a_1}{a}m \rfloor \ge m_1 \ge \lceil \frac{r_1-s_1^*}{r_1} m \rceil;  $$
$$ \lfloor \frac{a_2}{a}m \rfloor \ge m_2 \ge \lceil \frac{r_2-s_2^*}{r_2} m \rceil.  $$

Since $a_1+a_2=a$, one has
$$ \lfloor \frac{a_1}{a} m \rfloor + \lfloor \frac{a_2}{a} m \rfloor= \left\{ \begin{array}{ll} m-1 & \text{ if } \frac{a_1}{a} m \not \in \bZ; \\ m & \text{ if } \frac{a_1}{a} m  \in \bZ. \end{array} \right.
$$
Now the assumption $a\nmid m$ together with $\gcd(a,a_1)=1$ show that it is $m-1$.

Suppose on the contrary that $m < \frac{r_1r_2}{d}$.
Note also that $$r_2s_1^*a+r_1s_2^*a=r_2+a_1r_1r_2+r_1+a_2r_1r_2=ad+ar_1r_2.$$
It follows that
\begin{align*} & \lceil \frac{r_1-s_1^*}{r_1} m \rceil+\lceil \frac{r_2-s_2^*}{r_2} m \rceil \ge  \lceil \frac{r_1-s_1^*}{r_1} m+\frac{r_2-s_2^*}{r_2} m  \rceil \\ &= \lceil (2-  \frac{r_2s_1^*+r_1s_2^*}{r_1r_2})m \rceil = \lceil m-\frac{md}{r_1r_2} \rceil=m,
\end{align*}
which is a contradiction.
\end{proof}

We will need the following easy observation.
\begin{lem}\label{easycA}
Keep the notation as above. The following holds.
 \begin{enumerate}
 \item If $z$ or $u \in f$, then $\ct(X,S) \ge 1$.
\item If $P\in X$ admits a weighted blow up of weight $(s_1,s_2,1,1)$ such that $s_1, s_2 \ge 2$ and $s_1+s_2=d$, then $\ct(X,S) \not \in (\frac{1}{2},1)$.
\end{enumerate}
\end{lem}

\begin{proof}
If $z $  (resp. $u \in f$) then $m=w(f) \le w(z)= a$  (resp. $m=w(f)=w(u)=1$) and hence $\ct(X,S)=\frac{a}{m} \ge 1$.

Suppose that $\ct(X,S) \in (\frac{1}{2},1)$.
Consider the weighted blow up of weight $w':=(s_1,s_2,1,1)$ such that $s_1, s_2 \ge 2$.
Let $m':=w'(f)$ be the weighted multiplicity. Note also that weighted discrepancy $a'=1$ in this situation. By Lemma \ref{lower} and definition of canonical threshold, one sees $m' \le \frac{m}{a}=\frac{1}{\ct(X,S)} <2$. Hence either $z$ or $u \in f$, a contradiction.
\end{proof}

\begin{prop}\label{cA2}
We have $\mathcal{T}_{3,cA}^{\textup{can}} \cap (\frac{1}{2},1) \subseteq \{\frac{1}{2}+\frac{1}{n}\}_{n \ge 3}$.
\end{prop}
\begin{proof}
Suppose that $\ct(X,S)=\frac{a}{m} \in \mathcal{T}_{3,cA}^{\textup{can}} \cap (\frac{1}{2},1)$ where $a\ge 5$ is the weighted discrepancy where $\mathcal{T}_{3,cA}^{\textup{can}}$ is defined on page 3.
Note that $a\nmid m$.

Without loss of generality, we assume $r_1\le r_2$. 
Suppose first that $d \ge 4$, then we consider weighted blow up with weight $w'=(d-2,2,1,1)$.  One reaches a contradiction by Lemma \ref{easycA}.

Suppose next that $d=3$, we consider $w'=(1,2,1,1)$. From Lemma \ref{lower} and Lemma \ref{mainlem}, the weighted blow up with weight $w'$ is a divisorial contraction and thus $$\frac{1}{2}<\ct(X,S)\leq\frac{a'}{m'},$$ where $a'$ (resp. $m':=w'(f)$) denotes the weighted discrepancy (resp. weighted multiplicity) with respect to $w'$. As $a'=1$, one sees $m'=1$. Let $\frak{m}'\in f$ be a monomial with $w'(\frak{m}')=m'=1$.  (1) of Lemma \ref{easycA} implies that either $\frak{m}'=x$ or $\frak{m}'=y$. Since $w'(y)=2$, $x=\frak{m}'\in f$. Hence $m=w(f) \le w(x)=r_1$. By Proposition \ref{cAub}, one has $3=d \ge r_2  \ge r_1$. Hence $\ct(X,S)=\frac{a}{m} \ge \frac{5}{3}$, a contradiction.

Finally let $d=2$.
We consider  $w'=(1,1,1,1)$ and denote by $a'$ (resp. $m'$) the weighted discrepancy (resp. weighted multiplicity) with respect to this $w'$. Above argument (in the case $d=3$) yields $a'=m'=1$ and either $x=\frak{m}'\in f$ or $y=\frak{m}'\in f$.
In particular, $m\leq \max\{r_1,r_2\}=r_2$ and hence $2=d \ge r_1$ by Proposition \ref{cAub}.

Suppose that $r_1=2$. Proposition \ref{cAub} implies $m=r_2=ad-r_1=2a-2$ and so $$\ct(X,S)=\frac{a}{2a-2}=\frac{1}{2}+\frac{1}{2a-2}.$$

It remains to consider $r_1=1$ (and hence $r_2=2a-1$). Take $w_3=(1,2a-3,a-1,1)$.
Thus by Lemma \ref{lower} and the inequalities in Lemma \ref{mainlem}, we have
$$ \lfloor \frac{a-1}{a} m \rfloor \ge m_3 \ge \lceil \frac{2a-3}{2a-1} m \rceil. $$

Since $ \frac{1}{2} < \frac{a}{m} < 1$, it follows that
$\lfloor \frac{a-1}{a} m \rfloor=m-2$. Also
$$\lceil \frac{2a-3}{2a-1} m \rceil= \lceil m-\frac{2}{2a-1}m\rceil= \left\{ \begin{array}{ll} m-1 & \text{ if } m<2a-1; \\ m-2 & \text{ if } m=2a-1. \end{array} \right. $$

Hence $m=2a-1$ and notice that  $$\ct(X,S)=\frac{a}{2a-1}=\frac{1}{2}+\frac{1}{4a-2}.$$
Therefore,
$\mathcal{T}_{3,cA}^{\textup{can}} \cap (\frac{1}{2},1) \subseteq \{\frac{1}{2}+\frac{1}{n}\}_{n \ge 3}$.
\end{proof}

Next, we consider the ascending chain condition for canonical thresholds.
It is known that $\mathcal{T}_{3,sm}^{\textup{can}}$ satisfies the ACC (See \cite[Theorem 1.7]{Stepanov}).
In fact, for each type of $*$, we will use the following similar argument which is a  generalization of that in \cite[Lemma 2.6]{Stepanov}\footnote{In the proof of \cite[Lemma 2.6]{Stepanov}, Stepanov cited a Russian article for first and second finiteness properties of the semigroup $\mathbb{Z}^n_{\geq 0}$. Please see \cite[\textsection 3]{Khov} for English version. See also \cite[Appendix A]{HLL22} suggested by the referees.}.  
Suppose on the contrary that $\mathcal{T}_{3, *}^{\textup{can}}$ is not an ACC set. That is,  there is an infinite increasing sequence $\ct_1<\ct_2<\ct_3<\cdots$ with each canonical threshold $\ct_k\in \mathcal{T}_{3,*}^{\textup{can}}$.
For each $k$, let $\ct_k=\ct(X_k,S_k)$ where $P_k\in X_k$ is of (fixed) type $*$ and
$S_k$ is an effective Weil divisor defined by the equation $f_k$ analytically locally near $P_k \in X_k$.

Now, each canonical threshold $\ct(X,S)_k$ is computed by some divisorial contraction $\sigma_k: Y_k \to X_k$.
The classification of divisorial contraction asserts that $\sigma_k$ is a weighted blow up with weight $w_k$ over $P_k\in X_k$. Let $a_k$ denote the weighted discrepancy of $\sigma_k$ and $m_k=n_kw_k(f_k)$ denote the weighted multiplicity where $n_k$ is the index of $P_k \in X_k$ defined in section 2.

\bigskip
\noindent
{\bf Assumption A.}
Fix a type $*$ among $cA, cA/n, cD$ or $cD/2$.  For each infinite increasing sequence $\{\ct_k\}$ in $\mathcal{T}_{3, *}^{\textup{can}}$, we are able to find integers $i < j$ and an auxiliary weight $w^i_j$ such that:
\begin{enumerate}
\item
 the  weighted multiplicities satisfies $w_i(f_i) \le w_i(f_j)$(see Remark \ref{embedding});
\item  $n_iw_i \preceq n_j w^i_j$;
\item  the weighted blow up $\sigma^i_j:Y^i_j\to X_j$ with weight $w^i_j$ over the point $P_j\in X_j$ has prime exceptional divisor, denoted by $E^i_j$. Then in this situation,  $K_{Y^i_j}=\sigma^{i*}_j K_{X_j}+\frac{a_i}{n_j}E^i_j$, where $a_i$ is the weighted discrepancy.
\end{enumerate}

\begin{prop}\label{assum} Fix a type $*$ among $cA, cA/n, cD$ or $cD/2$. Suppose that Assumption A holds. Then $\mathcal{T}_{3,*}^{\textup{can}}$ satisfies the ascending chain condition.
\end{prop}

\begin{proof}
Fix any increasing sequence $\{\ct_k\}$ in $\mathcal{T}_{3, *}^{\textup{can}}$, and suppose that Assumption A holds.
Combining (1) with (2), we have the following: $$n_i w_i(f_i) \leq n_i w_i(f_j)
\le n_j w^i_j(f_j).$$
Moreover by (3), $E^i_j$ defines a valuation on $X_j$ and computation on $E^i_j$ shows that $\frac{a_i}{n_jw^i_j(f_j)} \ge \ct(X_j, S_j)=\ct_j$.
Then we have \[\ct_i=\frac{a_i}{n_i w_i(f_i)}
\ge \frac{a_i}{n_j w^i_j(f_j)}\ge \ct_j,\]
which is the desired contradiction.
\end{proof}

\begin{rem}\label{embedding}
Once the types of points $P_i$ and $P_j$ are the same with index $n_i=n_j$ such that one of the following holds:
\begin{itemize}
\item $n_i=1$;
\item $n_i>1$, and $b_{iq}=b_{jq}$ for all $q=1,...,4$ (resp. $q=1,...,5$) where for $k=i,j$, $P_k\in X_k$ is defined by the equation $\varphi_k=0/\frac{1}{n_k} (b_{k1}, b_{k2}, b_{k3}, b_{k4})$ (resp. $\varphi_{k1}=\varphi_{k2}=0/\frac{1}{n_k} (b_{k1}, b_{k2}, b_{k3}, b_{k4},b_{k5}$)), 
\end{itemize}
it make sense to consider the weight $w_i$ on $X_j$ and define the multiplicity $w_i(f_j)$ and compare the Newton diagrams $\Gamma^+(f_i)$ and $\Gamma^+(f_j).$
\end{rem}

\begin{prop}\label{cAACC} The Assumption A holds for $cA$. Hence
$\mathcal{T}_{3,cA}^{\textup{can}}$ satisfies the ascending chain condition.
\end{prop}
\begin{proof}
Fix  an infinite increasing sequence $\ct_1<\ct_2<\ct_3<\cdots$ with each $\ct_k\in \mathcal{T}_{3,cA}^{\textup{can}}$.
For each $k$, let $\ct_k=\ct(X_k,S_k)$ where $P_k\in X_k$ is defined by  $\varphi_k=xy+g_k(z,u)$ and $S_k$ is defined by $f_k$ analytically and locally.

Note that each $\ct(X_k,S_k)$ is realized by the divisorial contraction $\sigma_k: Y_k \to X_k$. Theorem 1.2(i) in \cite{Kawakita05} implies $\sigma_k$ is a weighted blow up of weight $w_k=wt(x,y,z,u)=(r_{k1},r_{k2},a_k,1)$ satisfying the following:
\begin{itemize}
\item $w_k(g_k(z,u))=r_{k1}+r_{k2}=a_kd_k$;
\item $z^{d_k}\in g_k(z,u)$ and hence $w_k(g_k)=w_k(z^{d_k})$;
\item $\gcd(r_{k1}, a_k)=\gcd(r_{k2}, a_k)=1$.
\end{itemize}

Passing to  subsequence, we may assume both sequences $\{a_k\}$ and $\{d_k\}$ are non-decreasing.
It follows from \cite[Lemma 2.5]{Stepanov} that we may assume the sequence of Newton polytopes  $\{\Gamma^+(f_k)\}$ is non-increasing.

We pick any $i <j$. Then we consider $(X_i,S_i)$ and $(X_j,S_j)$ and weights $w_i=(r_{i1},r_{i2},a_i,1), w_j=(r_{j1}, r_{j2},a_j,1)$.
We will consider the auxiliary weight $w^i_j=(r_{i1}, a_id_j-r_{i1},a_i,1)$. Clearly, $w_i \preceq w^i_j$.
Also, since $\Gamma^+(f_i) \supseteq \Gamma^+(f_j)$, one has $w_i(f_i) \le w_i (f_j)$. Note that $w_i (f_j)$ is well-defined since $P_i\in X_i$ has index one.

It remains to check the last condition of the Assumption A.
 Take the weighted blow up with weight $w^i_j$ over $X_j \ni P_j$. By Lemma \ref{lower}, its exceptional set is a prime divisor and the weighted discrepancy is $a_i$.

Therefore the Assumption A and hence the ACC holds by Proposition \ref{assum}.
\end{proof}

\section{canonical thresholds in $\mathcal{T}_{3, cA/n}^{\textup{can}}$}
In this section, we consider the canonical thresholds over $cA/n$ points. In viewing $\dagger_1$ and $\dagger_2$,  it is sufficient to consider divisorial contractions which are weighted blowups with $a \ge 5$. Hence we will assume that $a \ge 5$ in this section.

To investigate the canonical threshold in $\mathcal{T}_{3, cA/n}^{\textup{can}}$, let the canonical threshold $\ct(X,S)$ be computed by a weighted blow up $\sigma: Y \to X$
over a $cA/n$ point $P \in X$ with weighted discrepancy $a \ge 5$.
Theorem 1.2(i) in \cite{Kawakita05} shows that there exists an analytical identification $P\in X \simeq o\in (\varphi \colon xy+g(z^n,u)=0)$ in $\hat{\mathbb{C}}^4/\frac{1}{n}(1,-1,b,0)$ where $o$ denotes the origin of $\hat{\mathbb{C}}^4/ \frac{1}{n}(1,-1,b,0)$ and $\sigma$ is a weighted blow up of weight $w=wt(x,y,z,u)=\frac{1}{n} (r_{1},r_{2},a,n)$ satisfying the following:

\begin{itemize}
\item $nw(\varphi)=r_1+r_2=adn$ where $r_1,r_2,a,d,n$ are positive integers.
\item $z^{dn}\in g(z^n,u)$.
\item $a\equiv br_1$ (mod $n$) and $0<b<n$.
\item $\gcd(b,n)=\gcd(\frac{a-br_1}{n},r_1)=\gcd(\frac{a+br_2}{n},r_2)=1$ (See \cite[Lemma 6.6]{Kawakita05}).
\end{itemize}

Note that the condition $\gcd(b,n)=1$ follows from the classification of three dimensional non-Gorenstein terminal singularities by \cite{Mori85}.
The same proof as in Lemma \ref{lower} yields the following:

\begin{lem} \label{lower5}
Keep the notation as above. Suppose that there is another weight  $w'=\frac{1}{n}(r'_1, r'_2, a',1)$ satisfying $r'_1+r'_2=a'dn$, $a' \equiv br'_1$ (mod $n$),  and $a' \leq a$. Let $\sigma': Y' \to X$ be the weighted blow up with weight $w'$. Then the exceptional set $E'$ of $\sigma'$ is a prime divisor.
\end{lem}
\begin{proof}
It is convenient to consider {\it truncated weight} $v=\frac{1}{n}(a,1)$ and $v'=\frac{1}{n}(a',1)$ on $\{z, u \}$. Note that $w(\varphi)=v(g)=ad$.
Since $v' \succeq \frac{a'}{a} v$, one has
$v'(g) \ge \frac{a'}{a} v(g) = a'd,$ and hence $w'(\varphi)=a'd$.

Suppose that $a' \le a$. One has \[ \varphi^{w'}=xy+g^{w'}(z^n,u),\]
where $g^{w'}(z^n,u)$ is the $w'$-weighted homogeneous part of $g(z^n,u)$. As $z^{dn} \in g^{w'}(z^n,u)$, one has $g^{w'}(z^n,u) \ne 0$ and hence $\varphi^{w'}$ is irreducible by Lemma \ref{irred}.(1). The exceptional set $E'$ of $\sigma'$ which is a hypersurface in $\bP(r_1', r_2', a',1)$ defined by $\varphi^{w'}$ is thus irreducible.
 This verifies the proof of Lemma \ref{lower5}.
\end{proof}

Again, we will employ the method developed in \cite{CJK15} to produce two weights $w_1, w_2$ complementary to each other to approximate the weight computing canonical threshold. The approach is basically the same however more subtle than that in Proposition \ref{cAub}.

\begin{prop} \label{cAnub} Keep the notation as above.
Suppose that $\ct(X,S)=\frac{a}{m}$ is realized over a $cA/n$ point $P \in X$. Suppose furthermore
$a\nmid m$.
Then $m \ge \frac{r_1 r_2}{dn^2}$.
\end{prop}
\begin{proof}
We keep notations as above in this section and also follow the construction as in \cite[3.5]{CJK14}. We put $s_1:=\frac{a-br_1}{n}$ and $s_2:=\frac{a+br_2}{n}$. Recall that the integer $s_i$ is relatively prime to $r_i$, $i=1,2$ and we have the following:
$$\left\{ \begin{array}{cccc} a&=br_1+ns_1; \\  1&=q_1r_1+s^*_1s_1; \\ a&=-br_2+ns_2; \\ 1&=q_2r_2+s^*_2s_2. \end{array} \right.$$
for some integer $0\le s^*_i<r_i$ and some integer $q_i$.

In order to construct new weights $w_1$ and $w_2$ below, we set
\[\delta_1:=-nq_1+bs^*_1,\ \ \delta_2:=-nq_2-bs^*_2,\]
where $a-\delta_i$ (resp. $\delta_i$) will be the weighted discrepancy of the weighted blow up with weight $w_1$ (resp. $w_2$) when $\delta_i>0$.

It follows that
$$\left\{ \begin{array}{cc} \delta_1r_1+n&=as_1^*, \\  \delta_2r_2+n&=as_2^*, \end{array} \right.$$
and either $\delta_1>0$ or $\delta_2>0$ by \cite[Claims 1,2 in 3.5]{CJK14}. Indeed, for $i=1,2$, we have
\begin{align*}
\delta_ir_i+n&=(-nq_i+(-1)^{i+1}bs_i^*)r_i+n=n(1-q_ir_i)+(-1)^{i+1}bs_i^*r_i\\
&=s_i^*(ns_i+(-1)^{i+1}br_i)=as_i^*.
\end{align*}
Suppose that $\delta_i=0$. Then $nq_i=(-1)^{i+1}bs_i^*$. Since $\gcd(b,n)=1$, there exists an integer $t$ with $s_i^*=tn$ and $q_i=(-1)^{i+1}tb$. This implies $1=ta$ which contradicts to the assumption $a\ge 5$.
If both $\delta_1<0$ and $\delta_2<0$, then $n=-\delta_ir_i+as_i^*\ge r_i$ for $i=1,2$. Then $2n\ge r_1+r_2=adn\ge 5n$ which is impossible.

Note that it follows from $\delta_ir_i=as_i^*-n<as_i^*<ar_i$ that each $\delta_i<a$ for each $i=1,2.$

By exchanging $x$ and $y$ (resp. $r_1$ and $r_2$) if necessary, we assume that  $\delta_1>0$. Note that if $r_1=1$, then $s_1^*=0, q_1=1$ and $\delta_1=-n<0$ which contradicts to $\delta_1>0$. Thus, $r_1>1$.

 We set
$$w_1=\frac{1}{n}(r_1-s_1^*, r_2-\delta_1 dn+s_1^*, a-\delta_1,n).$$
$$w_2:=\frac{1}{n}(s_1^*, \delta_1 dn-s_1^*, \delta_1, n).$$
Note that $w_1\succeq \frac{r_1-s_1^*}{r_1} w$ and $w_2\succeq \frac{\delta_1 dn-s_1^*}{r_2} w$.
Lemma \ref{lower5} and inequalities in Lemma \ref{mainlem} give
$$ \lfloor \frac{a-\delta_1}{a}m \rfloor \ge m_1 \ge \lceil \frac{r_1-s_1^*}{r_1} m \rceil;   \text{ \ \ and } \lfloor \frac{\delta_1}{a}m \rfloor \ge m_2 \ge \lceil \frac{\delta_1 dn-s_1^*}{r_2} m \rceil.  $$

Since $b$ and $n$ are coprime and $a\nmid m$ where $m$ is an integral combination of $r_1,r_2,a,n$,
one has that $a\nmid \delta_1m$. One sees that
$$\lfloor \frac{a-\delta_1}{a}m \rfloor +  \lfloor \frac{\delta_1}{a}m \rfloor=m-1.$$

Suppose on the contrary that $m < \frac{r_1r_2}{dn^2}$. Then it follows that

\begin{align*} &\lceil \frac{r_1-s_1^*}{r_1} m \rceil+\lceil \frac{\delta_1 dn-s_1^*}{r_2} m \rceil  \ge  \lceil \frac{r_1-s_1^*}{r_1} m + \frac{\delta_1 dn-s_1^*}{r_2} m \rceil  \\
&= \lceil \frac{r_1r_2-s_1^*(adn)+r_1\delta_1dn}{r_1r_2} m\rceil = \lceil \frac{r_1r_2-dn^2}{r_1r_2}m \rceil=m,
\end{align*}

 which is a contradiction.
\end{proof}

Similar to Lemma \ref{easycA}, we have the following easy observation.

\begin{lem}\label{easycAn}
Keep the notation as above. Then the following holds.
\begin{enumerate}
\item If $z$ or $u\in f$, then $\ct(X,S) \ge 1$.
\item If $P\in X$ admits a weighted blow up of weight $\frac{1}{n}(s_1,s_2,1,n)$ such that $s_1, s_2 \ge 2$ and $s_1+s_2=dn$ and $1\equiv s_1b$ (mod $n$), then $\ct(X,S) \not \in (\frac{1}{2},1)$.
\end{enumerate}
\end{lem}

\begin{proof}
If $z$  (resp. $u \in f$) then the weighted multiplicity satisfies $m=nw(f) \le nw(z)=a$  (resp. $m=nw(f)=nw(u)=1$). So $\ct(X,S)=\frac{a}{m} \ge 1$.

Suppose that $\ct(X,S) \in (\frac{1}{2},1)$.
Consider the weighted blow up of weight $w':=\frac{1}{n}(s_1,s_2,1,1)$ such that $s_1, s_2\ge 2$ are two integers.
Let $m':=nw'(f)$ be the weighted multiplicity. Recall that weighted discrepancy $a'=1$ in this situation. By Lemma \ref{lower5}, we have $m' \le \frac{m}{a}=\frac{1}{\ct(X,S)} <2$. This gives either $z$ or $u \in f$, a contradiction.
\end{proof}

\begin{prop}\label{cAn2}
We have $\mathcal{T}_{3, cA/n}^{\textup{can}} \cap (\frac{1}{2}, 1)\subseteq \{\frac{1}{2}+\frac{1}{t}\}_{t\ge 3}$.
\end{prop}

\begin{proof}[Proof of Proposition \ref{cAn2}]
Suppose  that $\ct(X,S)=\frac{a}{m} \in (\frac{1}{2},1)$ and $a\ge 5$.
Note that $a\nmid m$.

\begin{claim}\label{b=1 or n-1, d=1 or 2}
$b\equiv\pm 1$ (mod $n$), \ $d=\left\{ \begin{array}{cc} 1,&  \textup{ if }n>2; \\  1 \textup{ or }2,&\textup{ if }n=2. \end{array} \right.$
\end{claim}
\begin{proof}[Proof of the Claim]
Suppose that $b \not \equiv \pm 1$ (mod $n$). Let $b^*<n$ be the smallest positive integer such that $bb^* \equiv 1$ (mod $n$). Then the weight $\frac{1}{n}(b^*, dn-b^*,1,n)$ has the property that $b^* \ge 2$ and $dn-b^* \ge 2$.
By Lemma \ref{easycAn}, $\ct(X,S) \not \in (\frac{1}{2},1)$.

We thus may assume that $b=1$, by symmetry\footnote{In the case $b=n-1$, we are able to obtain the same results by exchanging $x$ and $y$ (resp. $r_1$ and $r_2$) and thus interchanging the integers $w^*(x)$ and $w^*(y)$ for every weight $w^*=w, w', w_3$ and $w_1'$ in the argument of Claims \ref{b=1 or n-1, d=1 or 2}, \ref{xf}, \ref{bound n} and \ref{d=1}.}, from now on.
Suppose now that either $n\ge 3, d\ge 2$ or $n=2, d \ge 3$. We have $dn-n \ge 3$.  We consider the weight $\frac{1}{n}(n+1, dn-n-1, 1, n)$ and by Lemma \ref{easycAn}, $\ct(X,S) \not \in (\frac{1}{2},1)$.
This verifies the claim.
\end{proof}
\begin{claim} \label{xf}
$x \in f$ (resp. $x$ or $y \in f$) if $dn > 2$ (resp. if $n=2, d=1$).
\end{claim}
\begin{proof}[Proof of the Claim]
Since $b=1$, we consider the weight $w':=\frac{1}{n}(1,dn-1, 1,n)$. 
The inequalities $\frac{1}{2} < \ct(X,S) \le \frac{1}{m'}$ yield weighted multiplicity $m':=nw'(f)=1$. By (1) of Lemma \ref{easycAn}, $x \in f$ if $dn-1>1$. If $dn-1=1$, then either $x$ or $y \in f$.
\end{proof}

\begin{claim}\label{bound n}
$n=2$.
\end{claim}
\begin{proof}[Proof of the Claim]
Suppose that $n\ge 3$. We consider the weight $w_3:=\frac{1}{n}(n+3,2n-3,3,n)$. 
The inequality $\frac{1}{2}<\frac{a}{m}\le \frac{3}{m_3}$ yields weighted multiplicity $m_3:=nw_3(f)\le 5$.
By Lemma \ref{easycAn}, $z,u \not \in f$,  hence the only monomial $\frak{m}$ with possibly $nw_3(\frak{m}) \le 5$ is $y$. Therefore, $y \in f$. Now both $x, y \in f$, which is a contradiction because $f$ is semi-invariant under the weight $w$.
\end{proof}
\begin{claim}\label{d=1}
$d=1$.
\end{claim}
\begin{proof}[Proof of the Claim]
Suppose now that $d=2$. Then  $x \in f$ by Claim \ref{xf}.
One has weighted multiplicity $m:=2w(f) \le 2w(x)=r_1$.
We next consider $w'_1=\frac{1}{2}(3,1,1,2)$. Denote by $m'_1:=2w_1'(f)$ the weighted multiplicity.
Again, by Lemma \ref{lower5}, the inequalities $\frac{1}{2} < \ct(X,S) \le \frac{1}{m'_1}$ hold and hence one has $m'_1=1$. By (1) of Lemma \ref{easycAn}, $y\in f$.
One sees that $m=2w(f) \le 2w(y)=r_2$.
From Proposition \ref{cAnub}, for all $i=1,2$,
$$\frac{r_1r_2}{8}=\frac{r_1r_2}{dn^2}\le m\le r_i.$$
We have $r_1,r_2\le 8$. Since $adn=r_1+r_2$, one sees $a\le 4$. This contradicts to $a\ge 5$. This verifies the claim.
\end{proof}

We thus assume that $n=2, d=1$ from now on. We have either $x \in f$ or $y \in f$.
By exchanging $x$ and $y$ (resp. $r_1$ and $r_2$) if necessary, we may assume that $r_1 \le r_2$.
Suppose that $y\in f$.
By Proposition \ref{cAnub}, we have
$$\frac{r_1r_2}{4}=\frac{r_1r_2}{dn^2}\le m\le r_2,$$
which implies that $r_1 \le 4$ and hence $r_2=2a-r_1 > r_1$.
Note that $x\not\in f$ by the assumption $a\ge 5$ and the same computation\footnote{Indeed, suppose $x\in f$.
By Proposition \ref{cAnub}, we have
$\frac{r_1r_2}{4}=\frac{r_1r_2}{dn^2}\le m\le r_1,$
which implies $r_1\le r_2\le 4$ and thus $adn=r_1+r_2\le 8$. This contradicts to the assumption $a\ge 5$.}.

Recall that $a\equiv br_1\equiv r_1$ (mod $2$). Since $r_1+r_2=adn$, $y\in f$ and $f$ is semi-invariant, we see $a\equiv r_1\equiv r_2\equiv m$ (mod $2$). We next consider the weighted blow up of weight $w_{a-2}=\frac{1}{2}(r_1, r_2-4, a-2, 2)$ and denote by $m_{a-2}:=2w_{a-2}(f)$ the weighted multiplicity.

\begin{claim}\label{m=r_2}
$m=r_2$. 
\end{claim}
\begin{proof}[Proof of the Claim] Since $y\in f$, one has $m=2w(f)\le2w(y)=r_2$.
Suppose on the contrary that $m<r_2$. Lemma \ref{mainlem} gives the inequalities
\[ m-\lceil \frac{2m}{a} \rceil =\lfloor \frac{a-2}{a}m\rfloor \ge m_{a-2}\ge \lceil \frac{r_2-4}{r_2}m\rceil=m-\lfloor \frac{4m}{r_2} \rfloor. \]
In particular, $\lfloor \frac{4m}{r_2} \rfloor\ge \lceil \frac{2m}{a} \rceil $.
From the assumption $\ct(X,S)=\frac{a}{m}\in (1/2,1)$ and the condition $m<r_2$, we have $\lfloor \frac{4m}{r_2} \rfloor=\lceil \frac{2m}{a} \rceil=3$ and hence $m_{a-2}=m-3$. Now, $y\in f$ where $f$ is semi-invariant. One has $m_{a-2}=m-3=2w(f)-3\equiv 2w(y)-3\equiv r_2-3$ (mod $2$). On the other hand, $m_{a-2}=2w_{a-2}(f)\equiv 2w_{a-2}(y)=r_2-4$ (mod $2$). This leads to a contradiction and thus Claim \ref{m=r_2} is verified.
\end{proof}

\noindent
{\bf Case 1. $r_1=4$.}
We have
$$r_2=\frac{r_1r_2}{dn^2}\le m\le r_2.$$
Hence $m=r_2=2a-4$ and $$\ct(X,S)=\frac{a}{2a-4}=\frac{1}{2}+\frac{1}{a-2}.$$

\noindent
{\bf Case 2. $r_1=3$.} \\
Since $a$ is odd, then $a$ is either $6t+1, 6t+3$ or $6t+5$.

\noindent
{\bf Subcase 2.1 $a=6t+1$}\\
Now $w=\frac{1}{2}(3, 12t-1, 6t+1, 2)$. We consider $w_1=\frac{1}{2}(1, 4t+1, 2t+1, 2)$. 
From $w_1\succeq \frac{1}{3}w$, one sees the weighted multiplicity $m_1:=2w_1(f)\ge \lceil \frac{1}{3}m\rceil =\lceil \frac{12t-1}{3}\rceil=4t$. Since $y\in f$ and $f$ is semi-invariant, it follows that $m_1=2w_1(f)$ and $4t+1=2w_1(y)$ have the same parity. In particular, $m_1\ge 4t+1$.
This leads to a contradiction that
$$4t+\frac{6t-1}{6t+1}=\frac{2t+1}{6t+1} m\ge m_1 \ge 4t+1.$$

\noindent
{\bf Subcase 2.2 $a=6t+5$}\\
Now $w=\frac{1}{2}(3, 12t+7, 6t+5, 2)$. We consider $w_2=\frac{1}{2}(2, 8t+6, 4t+4, 2)$. 
One sees $m_2\ge \lceil \frac{2}{3}m\rceil =\lceil \frac{2(12t+7)}{3}\rceil=8t+5$.
Since $y\in f$ and $f$ is semi-invariant, $m_2=2w_2(f)$ and $8t+6=2w_2(y)$ have the same parity. In particular, $m_2\ge 8t+6$.
This leads a contradiction that
$$8t+6-\frac{2}{6t+5}=\frac{4t+4}{6t+5} m\ge m_2 \ge 8t+6.$$

\noindent
{\bf Subcase 2.3 $a=6t+3$}\\
Now $m=12t+3$ and hence
$$\ct(X,S)=\frac{a}{m}=\frac{2t+1}{4t+1}=\frac{1}{2}+\frac{1}{2(4t+1)}.$$

\noindent
{\bf Case 3. $r_1=2$.}\\
Then $$\ct(X,S) = \frac{a}{2a-2}=\frac{1}{2}+\frac{1}{2a-2}.$$

\noindent
{\bf Case 4. $r_1=1$.}\\
Then $$\ct(X,S) = \frac{a}{2a-1}=\frac{1}{2}+\frac{1}{4a-2}.$$
This completes the proof of Proposition \ref{cAn2}.
\end{proof}

As a generalization of Claim \ref{bound n}, we have the following.

\begin{lem}\label{boundindex}
Suppose that $\ct(X,S) \in \mathcal{T}_{3,cA/n}^{\textup{can}} \cap  \left( \frac{1}{k}, \frac{1}{k-1} \right)$ where $k\ge 2$ is a positive integer. Then $n \leq 3k$.
\end{lem}
\begin{proof}
Fix $\ct(X,S)=\frac{a}{m} \in \mathcal{T}_{3,cA/n}^{\textup{can}} \cap  \left( \frac{1}{k}, \frac{1}{k-1} \right)$.
Suppose on the contrary that  $n>3k$. Recall that $\gcd(b,n)=1$ by the classification of three dimensional non-Gorenstein singularities in \cite{Mori85}. We are able to consider the weights $w=\frac{1}{n}(r_1,r_2,a,n)$ and $w_3=\frac{1}{n}(s_1',s_2',3,n)$ satisfying $r_1+r_2=adn$, $s_1'+s_2'=3dn$, $\min\{s_1',s_2'\}>n$ and $a\equiv br_1$ (mod $n$) and $3\equiv bs_1'$ (mod $n$). Denote by $m_3:=nw_3(f)$ the weighted multiplicity. Note that exceptional set of the weighted blow up of weight $w_3$ is a prime divisor (cf. Lemma \ref{lower5}). It follows that $\frac{3}{m_3}\ge\ct(X,S)>\frac{1}{k}$. The condition $$\min\{s_1',s_2'\}>n>3k>m_3$$ implies that there exists $z^l\in f$ with $l \le k-1$ such that  $m_3=n w_3(z^l)$. This implies in particular that, $$a(k-1) \ge al = nw(z^l) \ge nw(f)=m,$$
which is a contradiction to $\frac{1}{k-1}>\ct(X,S)=\frac{a}{m}$.
\end{proof}

\begin{prop}\label{cAnACC} The Assumption A holds for $cA/n$. Hence
$\mathcal{T}_{3,cA/n}^{\textup{can}}$ satisfies the ascending chain condition.
\end{prop}
\begin{proof}
Suppose  that there is an infinite increasing sequence $\ct_1<\ct_2<\ct_3<\cdots$ with each $\ct_k\in \mathcal{T}_{3,cA/n}^{\textup{can}}$.
For each $k$, let $\ct_k=\ct(X_k,S_k)$ such that there exists an analytical identification $P_k\in X_k\simeq o_k\in (xy+g_k(z^{n_k},u)=0)$ in $\hat{\mathbb{C}}^4/\frac{1}{n_k}(b_k,-b_k,1,0)$ where $o_k$ denotes the origin of $\hat{\mathbb{C}}^4/\frac{1}{n_k}(b_k,-b_k,1,0)$ and $S_k$ is defined by $\{f_k=0\}/\frac{1}{n_k}(b_k,-b_k,1,0)$ analytically locally near $P_k \in X_k$. Now, each $\ct(X_k,S_k)$ is realized by a weighted blow up $\sigma_k:Y_k\to X_k$ with weight $w_k=wt(x,y,z,u)=\frac{1}{n_k}(r_{k1},r_{k2},a_k,n_k)$ satisfying the following:
\begin{itemize}
\item $n_kw_k (\varphi_k)=a_kd_kn_k=r_{k1}+r_{k2}$ where $r_{k1},r_{k2},a_k,d_k,n_k$ are positive integers;
\item $z^{d_kn_k }\in g_k(z^{n_k},u)$;
\item $a_k \equiv b_kr_{k1}$ (mod $n_k$);
\item $\gcd(b_k,n_k)=\gcd(\frac{a_k-b_kr_{k1}}{n_k},r_{k1})=\gcd(\frac{a_k+b_kr_{k2}}{n_k},r_{k2})=1$ and $0<b_k<n_k$.
\end{itemize}

By Lemma \ref{boundindex} and passing to a subsequence, we may assume that for all $k$, $n_k=n$ and $b_k=b$ for some integers $b$ and $n$. In particular, $f_j$ is semi-invariant under the weight $w_i$ for all $i<j$ and hence $w_i(f_j)$ is well-defined.

Passing to  subsequence, we may assume both sequences $\{a_k\}$ and $\{d_k\}$ are non-decreasing.
By \cite[Lemma 2.5]{Stepanov}, one may assume the sequence of Newton polytopes  $\{\Gamma^+(f_k)\}$ is non-increasing.
We pick any $i <j$ 
and consider the auxiliary weight $w^i_j=\frac{1}{n}(r_{i1}, a_id_jn-r_{i1},a_i,1)$. Clearly, $w_i \preceq w^i_j$.
Also, since $\Gamma^+(f_i) \supseteq \Gamma^+(f_j)$, one has $w_i(f_i) \le w_i (f_j)$.  

By Lemma \ref{lower5}, the weighted blow up with weight $w^i_j$ over $X_j \ni P_j$ has a prime exceptional divisor and has weighted discrepancy  $a_i$.
Therefore the Assumption A and hence the ACC holds by Proposition \ref{assum}.
\end{proof}

\section{canonical thresholds in $\mathcal{T}_{3, cD}^{\textup{can}}$}
In this section, we consider the canonical thresholds over $cD$ points with the weighted discrepancies $\ge 5$.

\begin{prop}\label{cD2}
We have $\mathcal{T}_{3, cD}^{\textup{can}} \cap (\frac{1}{2}, 1) = \emptyset$.
\end{prop}

\begin{proof}
Suppose on the contrary that there is a $\ct(X,S)=\frac{a}{m} \in (\frac{1}{2},1)$ with $a\ge 5$. We shall reach a contradiction.
Note that $a\nmid m$ and $\sigma$ is classified by case 1 and case 2. (cf. \cite [Theorem 1.2]{Kawakita05}).

\noindent
{\bf Case 1.}  Suppose $\sigma$ is a weighted blow up $\sigma: Y \to X$ with weight $w=wt(x,y,z,u)=(r+1,r,a,1)$ with center $P\in X$ by the analytical identification:
\[ (P\in X)\simeq  o\in ( \varphi: x^2+xq(z,u)+y^2u+\lambda y z^2+\mu z^3+p(y,z,u)=0) \subset \hat{\mathbb{C}}^4,\]
where $o$ denotes the origin of $\hat{\mathbb{C}}^4$ and
\begin{itemize}
\item $2r+1=ad$ where $d\ge 3$ and $a$ is odd;
\item $w(\varphi)=w(y^2u)=2r+1$; furthermore, $w(xq(z,u))=2r+1$ if $q(z,u)\neq 0$;
\item $p(y,z,u)\in (y,z,u)^4$.
\end{itemize}

\begin{claim}\label{z^dcD}
$z^d \in \varphi$.
\end{claim}
\begin{proof}[Proof of the Claim] 
Suppose that $z^d \not\in \varphi$. Then the origin of the $U_z$-chart on $Y$, denoted by $P_z$, is a point of index $a \ge 5$. By the classification of terminal singularity, $P_z$ is either a terminal quotient point or a terminal $cA/a$ point. 
If $P_z$ is a $cA/a$ point, then either $a|r$ or $a|r+1$ by the classification. If $P_z$ is a terminal quotient point, then there exists an integer $i$ or $j$ with $xz^i\in \varphi^w$ or $yz^j\in \varphi^w$. In particular, $r+ai=w(xz^i)=2r+1$ or $r+1+aj=w(yz^j)=2r+1$. It follows that $a|r$ or $a|r+1$. 
Together with $2r+1=ad$, one finds that $a=1$, which contradicts to the assumption that $a\ge 5$. Hence we have verified the argument of Claim \ref{z^dcD} (See also \cite[Case Ic]{CJK15}).
\end{proof}

Similar to Lemma \ref{lower}, we have the following:
\begin{lem} \label{cDlower}
Keep the notations as above. Suppose that there is another weight $w'=(r'+1, r', a',1)$ satisfying $2r'+1=a'd$ and $a' \leq a$. Let $\sigma': Y' \to X$ be the weighted blow up with weight $w'$. Then the exceptional set $E'$ of $\sigma'$ is a prime  divisor.
\end{lem}

\begin{proof}
Note that if $a'=a$ then $w=w'$ and $\sigma=\sigma'$. We may assume that $a'<a$.
As in Lemma \ref{lower}, it is convenient to consider {\it truncated weight} $v=(a,1)$ and $v'=(a',1)$ on $\{z, u \}$.  We have  $v' \succeq \frac{a'}{a} v$. Also, $v'(z^ju^k) = \frac{a'}{a} v(z^ju^k)$ if and only if $k=0$.

For any $\frak{m} \in q(z,u)$, $w(x \frak{m}) \ge 2r+1$ implies that $w(\frak{m}) \ge r$ and
$$ v'(\frak{m}) \ge \frac{a'}{a} v(\frak{m}) \ge \frac{a'}{a} r = \frac{2r'+1}{2r+1} r  > r'.$$
Hence $w'(xq(z,u)) > 2r'+1$.

Next we consider  $\frak{m}=y^iz^{j}u^{k}\in \varphi$. \\
For $i \ge 2$, one has $w'(\frak{m}) \ge 2 r'+1$ since $w(\frak{m}) \ge 2r+1$. For $i=1$, we have the following.
\begin{claim}\label{yzk}
Suppose $\frak{m}=yz^{j}u^{k}\in \varphi$. Then $w'(\frak{m})\ge 2r'+1$. Also, $w'(\frak{m})>2r'+1$ if $k\ge 1$.
\end{claim}
\begin{proof}[Proof of the Claim] For $\frak{m}=yz^{j}u^{k}\in \varphi$, 
$$ w'(\frak{m}) = r'+ v'(z^ju^k)  \ge r'+ \frac{a'}{a} v(z^ju^k)  \ge  r'+\frac{a'}{a} (r+1)  = 2r'+1+\frac{(r'-r)}{2r+1} >2r'.$$
Suppose now $k\ge 1$. If $w(\frak{m})>2r+1$, then $$ w'(\frak{m}) = r'+1+ v'(z^ju^{k-1})  \ge r'+1+ \frac{a'}{a} v(z^ju^{k-1})  \ge  r'+1+\frac{a'}{a} (r+1)>2r'+1.$$
Suppose on the contrary that $w(\frak{m})=2r+1$ and $w'(\frak{m})=2r'+1$. Then $aj+k=v(z^ju^k)=r+1$ and $a'j+k=v'(z^ju^k)=r'+1$. This gives
$(a-a')j=r-r'=(a-a')d/2$. This is absurd since it is impossible that $2aj=ad=2r+1$ where $a,j$ and $r$ are integers. This completes the argument of Claim \ref{yzk}.
\end{proof}

For $i=0$, $ w'(\frak{m}) =  v'(z^ju^k)  \ge  \frac{a'}{a} v(z^ju^k)  \ge  \frac{a'}{a} (2r+1)  = 2r'+1$.
Therefore, $w'(\varphi)=2r'+1$ and 
$\varphi^{w'}=y^2u+\eta yz^l+z^d$ for  some constant $\eta\in \bC$ where $l=(r'+1)/a'$ is an integer.

Note that $\eta=0$ if $a'\ge 3$. Indeed, it follows from $a'\ge 3$ and the equation $a'd=2r'+1$ that $a'\nmid r'$. In particular, $w'(yz^l)=r'+a'l\neq a'd=2r'+1$ since $l$ is an integer.  

The exceptional set $E'$ of $\sigma'$ is equal to
\[ \{y^2u+\eta yz^l+z^d=0\}\subset \mathbb{P}(r'+1,r',a',1).\]

Then, by Lemma \ref{irred}.(1), $(\varphi^{w'})$ is a prime ideal in the graded polynomial ring $\bC[x,y,z,u]$ and hence $E'$ is a prime divisor of $Y'$. We complete the proof of Lemma \ref{cDlower}.
\end{proof}

\begin{claim}\label{cDdm0}
$d=3$ and $m\le r$.
\end{claim}

\begin{proof}[Proof of the Claim]
Let $s=\frac{d-1}{2}$ and $\sigma':Y'\to X$ be the weighted blow up of weight $w'=(s+1,s,1,1)$. 
It follows from Lemma \ref{cDlower} that the exceptional set  of $\sigma'$ is a prime divisor.
Note that $\sigma'$ has weighted discrepancy $a'=1$.
By Lemma \ref{mainlem},  one sees the weighted multiplicity $m'=w'(f)=1$ where the Weil divisor $S$ is given by $f=0$.

If $d\ge 5$, then $z \in f$ or $u \in f$, and hence $m \le a$. We reach a contradiction that $1> \ct(X,S)=\frac{a}{m}\ge 1$.

If $d=3$, we have $y\in f$. So $m\le w(y)=r$.
\end{proof}
Now $d=3$.
We then consider the weighted blow up $\sigma_1\colon Y_1\to X$ (resp. $\sigma_2\colon Y_1\to X$) with weight $w_1=(3,3,2,1)$ (resp.  $w_2=(r-2,r-3,a-2,1)$). The argument of Lemma \ref{cDlower} also implies the following.
\begin{claim}\label{irred(3,3,2,1)}
$\varphi^{w_1}=x^2+z^3$.
\end{claim}
\begin{proof}[Proof of the Claim]
To see this, we set $v=(a,1)$ and $v_1=(2,1)$ on $\{z, u \}$ as above. As $a\ge 5>2$,  one sees $v_1(z^ju^k)\ge \frac{2}{a} v(z^ju^k)$ and the last equality holds if and only if $k=0$.

For any $0\neq \frak{m}=z^ju^k \in q(z,u)$, we have $w_1(x\frak{m})>6$. Suppose not. One sees $w_1(\frak{m})=2j+k\leq 3$. Together with $aj+k=v(\frak{m})=w(\frak{m}) \ge r= (3a-1)/2\ge 7$, we have $j=1$ and so $(2r+1+3k)/3=a+k=w(\frak{m}) \ge r$. This implies $4\ge 1+3k\ge r\ge 7$, which is impossible.
Thus, $w_1(xq(z,u)) > 6$.

It is clear that $w_1(y^2u)=w_1(yz^2)=7>6=w_1(x^2)=w_1(z^3).$

Next we consider  $\frak{m}=y^iz^{j}u^{k}\in \varphi$. \\
For $i \ge 2$, one has $w_1(\frak{m}) \ge 7$ since $w(\frak{m}) \ge 2r+1$. \\
For $i=1$,
$$ w_1(\frak{m}) = 3+ v_1(z^ju^k)  \ge 3+ \frac{2}{a} v(z^ju^k)  \ge  3+\frac{2}{a} (r+1) >3+d=6.$$
For $i=0$, $ w_1(\frak{m}) =  v_1(z^ju^k)  \ge  \frac{2}{a} v(z^ju^k)  \ge  \frac{2}{a} (2r+1)=6$.
Thus, $w_1(\varphi)=6$ and we have verified Claim \ref{irred(3,3,2,1)}.
\end{proof}

It is easy to see that $x^2+z^3$ is irreducible
hence the exceptional set of $\sigma_1$ is a prime divisor.
By Lemma \ref{cDlower}, the exceptional set of  $\sigma_2$ is a prime divisor (see also \cite[Case Ic]{CJK15}).

Recall that $2r+1=3a$ and hence
\[w_1\succeq \frac{3}{r+1} w\ \ \textup{ and } \ \ \ w_2\succeq \frac{r-3}{r} w.\]
It follows from the inequalities in Lemma \ref{mainlem} that \[\lfloor \frac{2}{a}m\rfloor \ge m_1\ge \lceil \frac{3}{r+1} m\rceil \ \ \textup{ and } \ \ \ \lfloor \frac{a-2}{a}m\rfloor \ge m_2\ge \lceil \frac{r-3}{r}m\rceil, \eqno{\dagger_6} \]
where $m_1:=w_1(f)$ and $m_2:=w_2(f)$ are the weighted multiplicities.
From Claim \ref{cDdm0} and assumption $a\ge 5$, one sees
$3m<r(r+1)$. Thus
\begin{align*}
& \lceil \frac{3}{r+1}m\rceil +\lceil \frac{r-3}{r}m\rceil \ge   \lceil \frac{3}{r+1}m +\frac{r-3}{r}m\rceil  = \lceil m-\frac{3m}{r(r+1)} \rceil=m.
\end{align*}

However, $a$ is odd and $a\nmid m$, hence  $\frac{2m}{a}$ is not an integer.
This implies that
\[\lfloor \frac{2}{a}m\rfloor+\lfloor \frac{a-2}{a}m\rfloor =m-1, \]
which contradicts to $\dagger_6$.

\noindent
{\bf Case 2.}  Suppose $\sigma$ is a weighted blow up with weight $w=(r+1,r,a,1,r+2)$ with center $P\in X$ by the analytical identification
$$o\in \left( \begin{array}{ll}
\varphi_{1} \colon x^2+yt+p(y,z,u)=0 ;\\
 \varphi_{2} \colon  yu+z^{d}+q(z,u)u+t=0 \\
\end{array} \right) \subset \hat{\mathbb{C}}^5
$$
where $o$ denotes the origin of $\hat{\mathbb{C}}^5$ such that
\begin{itemize}
\item $r+1=ad$ where $d\ge 2$;
\item $w(\varphi_{1})=2(r+1)$;
\item $w(\varphi_{2})=r+1$; moreover, $w(q(z,u)u)=r+1$ if $q\neq 0$.
\end{itemize}

Let $\sigma_1:Y_1\to X$  be the weighted blow up with weight $w_1=(d,d,1,1,d)$. By Lemma \ref{(d,d,1,1,d)} below, $\sigma_1$ is a divisorial contraction. Moreover, $\sigma_1$ has weighted discrepancy $1$.
One has $\frac{1}{m_1} \ge \ct(X,S) >\frac{1}{2}$ where $m_1:=w_1(f)$. Hence $m_1=1$.

By $d\ge 2$, we see $z\in f$ or $u\in f$. This implies $m\leq a$ and thus $\ct(X,S)\ge 1$, a contradiction.

Therefore, if $\ct(X,S) \in \mathcal{T}_{3, cD}^{\textup{can}}$ then $ \ct(X,S) \not \in (\frac{1}{2}, 1)$.
We finish the proof of Proposition \ref{cD2}.
\end{proof}

\begin{lem} \label{(d,d,1,1,d)}
Keep the notations as in case 2 above. Let $\sigma_1: Y_1 \to X$ be the weighted blow up with weight $w_1=(d,d,1,1,d)$. Then the exceptional set $E_1$ of $\sigma_1$ is a prime divisor of $Y_1$.
\end{lem}

\begin{proof}
As in Lemma \ref{lower}, it is convenient to consider {\it truncated weight} $v=(a,1)$ and $v_1=(1,1)$ on $\{z, u \}$.  We have  $v' \succeq \frac{1}{a} v$ and also $v_1(z^ju^k)=\frac{1}{a} v(z^ju^k)$ if and only if $k=0$.

First we consider $\frak{m}=y^iz^j u^k \in p(y,z,u)$ in $\varphi_1$.\\
For $i \ge 2$, one has $w_1(\frak{m}) >2d$ since $w(\frak{m}) \ge 2r+2$.
For $i = 1$,
$$w_1(\frak{m})=d+v_1(z^ju^k) \ge d +\frac{1}{a} v(z^ju^k) \ge d+ \frac{1}{a}(r+2)>d+\frac{r+1}{a}= 2d.$$
For $i=0$, $w_1(\frak{m})=v_1(\frak{m}) \ge \frac{1}{a} v(\frak{m} )  \ge \frac{1}{a} (2r+2)=2d$ and equality holds if and only if $\frak{m}=z^{2d}$.
Hence $w_1(\varphi_1)=2d$.

It is clear that $w_1(yu)=d+1>d$. Also for any $\frak{m} \in q(z,u)u$ in $\varphi_2$,
$$w_1(\frak{m})=v_1(\frak{m}) > \frac{a'}{a} v(\frak{m}) \ge \frac{1}{a}(r+1)=d.$$
Hence $w_1(\varphi_2)=d$.
One sees that $\varphi_1^{w_1}=x^2+yt+ \lambda z^{2d},\ \varphi_2^{w_1}=z^d+t$ for some constants $\lambda,\mu$. Note that $\lambda$ may be zero (for instance, see \cite[Example 8.2]{Kawakita05}). Then
\begin{align*}
E_1&=\{x^2+yt+ \lambda z^{2d}=z^d+t=0\}\subset\mathbb{P}(d, d, 1,1, d)\\
&\simeq \{x^2-yz^d+ \lambda z^{2d}=0\}\subset\mathbb{P}(d,d,1,1).
\end{align*}
Now, Lemma \ref{irred}.(1) implies that $(x^2-yz^d+ \lambda z^{2d})$ is a prime ideal of the graded polynomial ring $\bC[x,y,z,u]$. Thus, $E'$ is a prime divisor of $Y'$. This verifies the proof of Lemma \ref{(d,d,1,1,d)}.
\end{proof}

To prove the ascending chain condition for $\mathcal{T}_{3,cD}^{\textup{can}}$, we need the following whose argument is similar to Lemma \ref{(d,d,1,1,d)}.
\begin{lem} \label{cDlower2}
Keep the notations as in case 2 above. Suppose that there is another weight $w'=(r'+1, r', a',1, r'+2)$ satisfying $r'+1=a'd$ and $a' \leq a$. Let $\sigma': Y' \to X$ be the weighted blow up with weight $w'$. Then the exceptional set of $\sigma'$ is a prime divisor.
\end{lem}

\begin{proof}
Note that if $a'=a$ then $w=w'$ and $\sigma=\sigma'$. We may assume that $a'<a$.
As in Lemma \ref{lower}, it is convenient to consider {\it truncated weight} $v=(a,1)$ and $v'=(a',1)$ on $\{z, u \}$.  We have  $v' \succeq \frac{a'}{a} v$.
Also, one sees $v'(z^ju^k)=\frac{a'}{a} v(z^ju^k)$ if and only if $k=0$.

First we consider $\frak{m}=y^iz^j u^k \in p(y,z,u)$ in $\varphi_1$.\\
For $i \ge 2$, one has $w'(\frak{m}) \ge 2r'+2$ since $w(\frak{m}) \ge 2r+2$.
For $i = 1$,
$$w'(\frak{m})=r' +v'(z^ju^k) \ge r' +\frac{a'}{a} v(z^ju^k) \ge r'+ \frac{a'}{a}(r+2)=2r'+2+\frac{r'-r}{r+1} > 2r'+1.$$
For $i=0$, $w'(\frak{m}) \ge \frac{a'}{a} v(\frak{m} ) \ge 2r'+2$.
Hence $w'(\varphi_1)=2r'+2$.

Also for any $\frak{m} \in q(z,u)u$ in $\varphi_2$,
$$w'(\frak{m}) \ge \frac{a'}{a} v(\frak{m}) \ge \frac{a'}{a}(r+1)=r'+1.$$

One sees that $\varphi_1^{w'}=x^2+yt+ \lambda z^{2d}+ \mu y^2u^2,\ \varphi_2^{w'}=yu+z^d$ for some constants $\lambda,\mu$. Note that $\lambda,\mu$ may be zero (for instance, see \cite[Example 8.2]{Kawakita05}). Then
\[E'=\{x^2+yt+ \lambda z^{2d}+ \mu y^2u^2=yu+z^d=0\}\subset\mathbb{P}(r'+1, r', a',1, r'+2).\]
Let $U_y=\{y\neq 0\}, U_u:=\{u\neq 0\}$ and $U_t:=\{t\neq 0\}$ be three affine open charts of $\mathbb{P}(r'+1, r', a',1, r'+2)$.
Note that
\[E'\cap U_y\simeq \{x^2+t+ \lambda z^{2d}+ \mu u^2=u+z^d=0\}\subset \hat{\mathbb{C}}_{x,z,u,t}^4/\frac{1}{r'}(1,a',1,2).\]
Since $x^2+t+ \lambda z^{2d}+ \mu y^2u^2$ and $u+z^d$ have linear term $t$ and $u$, one sees $E'\cap U_y\simeq \hat{\mathbb{C}}_{x,z}^2/\frac{1}{r'}(1,a')$ which is a prime divisor. Next, one has
\begin{align*}
E'\cap U_u&\simeq \{x^2+yt+ \lambda z^{2d}+ \mu y^2u^2=y+z^d=0\}\subset \hat{\mathbb{C}}_{x,y,z,t}^4\\
&= \{x^2+yt+ z^{2d}(\lambda + \mu u^2)=y+z^d=0\}\subset \hat{\mathbb{C}}_{x,y,z,t}^4\\
&\simeq \{x^2-z^dt+z^{2d}(\lambda + \mu u^2)=0\}\subset \hat{\mathbb{C}}_{x,z,t}^3,
\end{align*}
which is a prime divisor by Lemma \ref{irred}.(1).
Also, $E'\cap U_t$ is isomorphic to
\begin{align*}
&\{x^2+y+ \lambda z^{2d}+ \mu y^2u^2=yu+z^d=0\}\subset \hat{\mathbb{C}}_{x,y,z,u}^4/\frac{1}{r'+2}(r'+1,r',a',1)\\
&=\{x^2+y+ (\lambda + \mu) z^{2d}=yu+z^d=0\}\subset \hat{\mathbb{C}}_{x,y,z,u}^4/\frac{1}{r'+2}(r'+1,r',a',1)\\
&\simeq \{ -(x^2+(\lambda + \mu) z^{2d})u+z^d=0\}\subset \hat{\mathbb{C}}_{x,z,u}^3/\frac{1}{r'+1}(r',a',1).
\end{align*}
By Lemma \ref{irred}.(2), $(-(x^2+(\lambda + \mu) z^{2d})u+z^d)$ is a prime ideal of the polynomial ring $\bC[x,z,u]$. The subset $\{ -(x^2+(\lambda + \mu) z^{2d})u+z^d=0\}\subset \hat{\mathbb{C}}_{x,z,u}^3$ is a prime divisor.
Hence $E'\cap U_t$ is a prime divisor of $U_t$ as well.

Since the three open subsets $E'\cap U_y$, $E'\cap U_u$ and $E'\cap U_t$ of $E'$ are prime divisors and $E'\cap U_y\cap U_u\cap U_t\neq \emptyset$ and $E'\cap \{y=u=t=0\}=\emptyset$, $E'$ is a prime divisor and this verifies Lemma \ref{cDlower2}.
\end{proof}

\begin{prop}\label{cDACC}  The Assumption A holds for $cD$. Hence
$\mathcal{T}_{3,cD}^{\textup{can}}$ satisfies the ascending chain condition.
\end{prop}
\begin{proof}
Suppose on the contrary that there is an infinite increasing sequence $\ct_1<\ct_2<\ct_3<\cdots$ with each $\ct_k\in \mathcal{T}_{3,cD}^{\textup{can}}$.
Note that each $\ct_k=\ct(X_k,S_k)$ is realized by the divisorial contraction $\sigma_k: Y_k \to X_k$.
Theorem 1.2(ii) in \cite{Kawakita05} indicates that $P_k\in X_k$ and $\sigma_k$ are described in case 1 and case 2.

\noindent
{\bf Case 1.} Suppose $P_k\in X_k$ is defined by the analytical identification
\[o\in (\varphi_k:= x^2+xq_k(z,u)+y^2u+\lambda_k y z^2+\mu_k z^3+p_k(y,z,u=0)\subset \hat{\mathbb{C}}^4,\] where $o$ denotes the origin of $\hat{\mathbb{C}}^4$
and each $\sigma_k$ is a weighted blow up of weight $w_k=wt(x,y,z,u)=(r_{k}+1,r_{k},a_k,1)$ satisfying the following:
\begin{itemize}
\item $2r_k+1=a_k d_k$ with $d_k \ge 3$;
\item $w_k(\varphi_k)=w_k(y^2u)=2r_k+1$; furthermore, $w_k(xq_k(z,u))=2r_k+1$ if $q_k\neq 0$.
\item $z^{d_k}\in \varphi_k$ with $w_k(z^{d_k})=w_k(\varphi_k)$ by Claim \ref{z^dcD}.
\end{itemize}

Passing to a subsequence, we may assume that the  sequences $\{a_k\}, \{d_k\}$ are non-decreasing. It follows from \cite[Lemma 2.5]{Stepanov} that we may assume the sequence of Newton polytopes  $\{\Gamma^+(f_k)\}$ is non-increasing.

We now pick some $i<j$ such that $a_i < a_j$. Let $s^i_j$ be  an integer such that $2s^i_j+1=a_id_j$ and pick $w^i_j:=(s^i_j+1,s^i_j,a_i,1)$.
Clearly, $w_i \preceq w^i_j$. Also, since $\Gamma^+(f_i) \supseteq \Gamma^+(f_j)$, one has $w_i(f_i) \le w_i(f_j)$. By Lemma \ref{cDlower}, the Assumption A holds.

\noindent {\bf Case 2.} Suppose $P_k\in X_k$ is defined by the analytical identification
$$o\in \left( \begin{array}{ll}
\varphi_{k1} \colon  x^2+yt+p_k(y,z,u)=0 ;\\
 \varphi_{k2} \colon yu+z^{d_k}+q_k(z,u)u+t=0. \\
\end{array} \right) \subset \hat{\mathbb{C}}^5,$$ where $o$ denotes the origin of $\hat{\mathbb{C}}^5$
and each $\sigma_k$ is a weighted blow up of weight $w_k=wt(x,y,z,u,t)=(r_{k}+1,r_{k},a_k,1,r_k+2)$ satisfying the following:
\begin{itemize}
\item $ r_k+1=a_kd_k$ with $d_k >1$;
\item $w_k(\varphi_{k1})=2(r_k+1)$;
\item $w_k(\varphi_{k2})=r_k+1$; moreover, $w_k(q_k(z,u)u)=r_k+1$ if $q_k\neq 0$.
\end{itemize}

Passing to a subsequence, we may assume that the  sequences $\{a_k\}, \{d_k\}$ are non-decreasing.
By \cite[Lemma 2.5]{Stepanov}, we may assume the sequence of Newton polytopes  $\{\Gamma^+(f_k)\}$ is non-increasing.

We pick some $i<j$ such that $a_i < a_j$.  Let $s^i_j$ be the integer with $s^i_j+1=a_id_j$ and take $w^i_j:=(s^i_j+1,s^i_j,a_i,1,s^i_j+2)$.
Clearly, $w_i \preceq w^i_j$ and $w_i(f_i) \le w_i(f_j)$. 
By Lemma \ref{cDlower2}, the Assumption A holds and this completes the proof of the proposition.
\end{proof}

\section{canonical thresholds in $\mathcal{T}_{3, cD/2}^{\textup{can}}$}
In this section, we consider the canonical thresholds over $cD/2$ points with the weighted discrepancies $\ge 5$.

\begin{prop} \label{cD22}
We have $\mathcal{T}_{3, cD/2}^{\textup{can}} \cap (\frac{1}{2}, 1) \subseteq \{\frac{1}{2}+\frac{1}{t}\}_{t\ge 18}$.
\end{prop}
\begin{proof}
Suppose  that there is a $\ct(X,S)=\frac{a}{m} \in (\frac{1}{2},1)$ with $a\ge 5$. 
Note that $a\nmid m$ and $\sigma$ is classified by two cases according to Theorem 1.2 in \cite{Kawakita05}.

\noindent
{\bf Case 1.}  Suppose $\sigma$ is a weighted blow up $\sigma: Y \to X$ with weight $w=\frac{1}{2}(r+2,r,a,2)$ with center $P\in X$ by the analytical identification:
\[ o\in (\varphi: x^2+xzq(z^2,u)+y^2u+\lambda y z^{2\alpha -1}+p(z^2,u)=0 ) \subset \hat{\mathbb{C}}^4/\frac{1}{2}(1,1,1,0)\]
where $o$ denotes the origin of $\hat{\mathbb{C}}^4/\frac{1}{2}(1,1,1,0)$ such that the following holds:
\begin{itemize}
\item $r+1=ad$ where both $a$ and $r$ are odd;
\item $w(\varphi)=w(y^2u)=r+1$; furthermore, $w(xzq(z^2,u))=r+1$ if $q\neq 0$;
\end{itemize}
\begin{claim}\label{z^2dcD2}
$z^{2d} \in p(z^2,u)$.
\end{claim}
\begin{proof}[Proof of the Claim] 
Suppose that $z^{2d} \not\in p(z^2,u)$. Note that the exceptional set $E=\{\varphi^w=0\}\subset \mathbb{P}(r+2,r,a,2)$ is a prime divisor and $y^2u\in \varphi^w$, so $u\nmid \varphi^w$.
Now, both $x^2\not\in \varphi^w$ and $z^{2d} \not\in p(z^2,u)$.
There must exist an integer $i$ with $xz^i\in \varphi^w$ or $yz^{2\alpha -1}\in \varphi^w$ where $\alpha$ is an integer. In particular, $r+2+ai=2w(xz^i)=2r+2$ or $r+a(2\alpha-1)=2w(yz^{2\alpha -1})=2r+2$.
Together with the equation $r+1=ad$, we see
$$a|\gcd(r,r+1)=1 \textup{ or }a|\gcd(r+2,r+1)=1$$ which contradicts to the assumption $a\ge 5$. Hence we have verified the argument of Claim \ref{z^2dcD2} (See also \cite[3.1]{CJK14}).
\end{proof}

Similar to Lemma \ref{lower}, we have the following:
\begin{lem} \label{cD2lower}
Keep the notations as above. Suppose that there is another weight $w'=\frac{1}{2}(r'+2, r', a',2)$ satisfying $r'+1=a'd$ and $a' \leq a$. Let $\sigma': Y' \to X$ be the weighted blow up with weight $w'$. Then the exceptional divisor $E'$ of $\sigma'$ is a prime divisor.
\end{lem}

\begin{proof}
Note that if $a'=a$ then $w=w'$ and $\sigma=\sigma'$. We may assume that $a'<a$.
As in Lemma \ref{lower}, it is convenient to consider {\it truncated weight} $v=\frac{1}{2}(a,2)$ and $v'=\frac{1}{2}(a',2)$ on $\{z, u \}$. We have  $v' \succeq \frac{a'}{a} v$. 
Also, one sees $v'(z^ju^k)=\frac{a'}{a} v(z^ju^k)$ if and only if $k=0$.

For any $\frak{m} \in q(z^2,u)$, $w(xz \frak{m}) \ge r+1$ implies that $v(z\frak{m}) \ge \frac{r}{2}$ and
$$ v'(z\frak{m}) \ge \frac{a'}{a} v(z\frak{m}) \ge \frac{a'}{a}\cdot\frac{r}{2}> \frac{r'}{2}.$$
Hence $w'(xzq(z^2,u)) > r'+1$.

Now for any $\frak{m}=z^{j}u^{k}\in \varphi$, one has
$w'(\frak{m}) \ge \frac{a'}{a} w(\frak{m}) \ge r'+1$. Finally, for $\frak{m}=y z^{2\alpha-1}$,
$w(\frak{m}) \ge r+1$ implies that $v(z^{2\alpha-1}) \ge \frac{r+2}{2}$.
Hence $$ w'(\frak{m}) =\frac{r'}{2}+ v'(z^{2\alpha-1})\ge \frac{r'}{2}+ \frac{a'}{a}v(z^{2\alpha-1})\ge \frac{r'}{2}+ \frac{r'+1}{r+1}\cdot\frac{r+2}{2} > \frac{2r'+1}{2}.$$

Therefore, $w'(\varphi)=r'+1$ and $\varphi^{w'}=y^2u+z^{2d}$ or $\varphi^{w'}=y^2u+\lambda y z^{2\alpha -1}+z^{2d}$.
The exceptional set $E'$ is equal to
\[\{\varphi^{w'}=0\}\subset \mathbb{P}(r'+2, r', a',2)=\textup{Proj}\ \hat{\mathbb{C}}[x,y,z,u], \]
where $\hat{\mathbb{C}}[x,y,z,u]$ is a graded polynomial ring with wt$(x,y,z,u)=(r'+2, r', a',2)$.
From Lemma \ref{irred}.(1), $(\varphi^{w'})$ is a prime ideal of  $\hat{\mathbb{C}}[x,y,z,u]$. Thus, $E'$ is a prime divisor and this verifies Lemma \ref{cD2lower}.
\end{proof}

\begin{claim}\label{cD2d=2mr}
$d=2$, $y\in f$ and $m\le r$.
\end{claim}
\begin{proof}[Proof of the Claim] 
Let $s=d-1$ and $\sigma':Y'\to X$ be the weighted blow up of weight $\frac{1}{2}(s+2,s,1,2)$.
It follows from  Lemma \ref{cD2lower} that $\sigma'$ has prime exceptional divisor.
By Lemma \ref{mainlem},  one sees the weighted multiplicity $m'=2w'(f)=1$ where the Weil divisor $S$ is given by $f=0$.
If $d\ge 3$, then $z \in f$, and hence $m \le a$. One has a contradiction that $1> \ct(X,S)=\frac{a}{m}\ge 1$.
If $d=2$, we have $y\in f$. So $m=w(f)\le w(y)=r$. Claim \ref{cD2d=2mr} is verified.
\end{proof}

We then consider weighted blow up with weight $w_{a-2}=\frac{1}{2}(r-2,r-4,a-2,2)$. By Lemma \ref{cD2lower}, the weighted blow up with weight $w_{a-2}$ has prime exceptional divisor. Denote by $m_{a-2}:=2w_{a-2}(f)$ the weighted multiplicity. 
The argument of Claim \ref{m=r_2} give the following. 
\begin{claim}\label{m=r}
$m=r$. 
\end{claim}
\begin{proof}[Proof of the Claim] 
Suppose on the contrary. By Claim \ref{cD2d=2mr}, $m<r$. 
By Lemma 2.1, \[ m-\lceil \frac{2m}{a} \rceil =\lfloor \frac{a-2}{a}m\rfloor \ge m_{a-2}\ge \lceil \frac{r-4}{r}m\rceil=m-\lfloor \frac{4m}{r} \rfloor. \]
In particular, $\lfloor \frac{4m}{r} \rfloor\ge \lceil \frac{2m}{a} \rceil $.
From the assumption $\ct(X,S)=\frac{a}{m}\in (1/2,1)$ and the condition $m<r$, we have $\lfloor \frac{4m}{r} \rfloor=\lceil \frac{2m}{a} \rceil=3$ and hence $m_{a-2}=m-3$. Now, $y\in f$ by Claim \ref{cD2d=2mr} and $f$ is semi-invariant. One sees $m_{a-2}=2w_{a-2}(f)\equiv 2w_{a-2}(y)\equiv r-4$ (mod $2$). On the other hand, $m_{a-2}=m-3=2w(f)-3\equiv 2w(y)-3\equiv r-3$ (mod $2$). This leads to a contradiction and Claim \ref{m=r} is verified.
\end{proof} 
Recall that $r=ad-1=2a-1$ by Claim \ref{cD2d=2mr}. It follows from Claim \ref{m=r} that  
\[\ct(X,S)=\frac{a}{m}=\frac{a}{r}=\frac{a}{2a-1}=\frac{1}{2}+\frac{1}{2(2a-1)}.\]

\noindent
{\bf Case 2.}  Suppose $\sigma$ is a weighted blow up with weight $w=\frac{1}{2}(r+2,r,a,2,r+4)$ with center $P\in X$ by the analytical identification:
$$o\in \left( \begin{array}{ll}
\varphi_{1}:= x^2+yt+p(z^2,u)=0 \\
 \varphi_{2}:= yu+z^{2d+1}+q(z^{2},u)zu+t=0 \\
\end{array} \right) \textup{ in } \hat{\mathbb{C}}_{x,y,z,u,t}^5/\frac{1}{2}(1,1,1,0,1),$$ where $o$ denotes the origin of $\hat{\mathbb{C}}_{x,y,z,u,t}^5/\frac{1}{2}(1,1,1,0,1)$ such that
\begin{itemize}
\item $r+2=a(2d+1)$ where $d$ is a positive integer;
\item $w(\varphi_{1})=r+2$;
\item $w(\varphi_{2})=\frac{r+2}{2}$; moreover, $w(q(z^2,u)zu)=\frac{r+2}{2}$ if $q(z^2,u)\neq 0$.
\end{itemize}

We consider the weighted blow up with weight $w'=\frac{1}{2}(2d+1,2d+1,1,2, 2d+1)$. It has prime exceptional divisor by Lemma \ref{1/2(2d+1,2d+1,1,2, 2d+1)} (see also \cite[3.2]{CJK14}).
One sees the weighted multiplicity $m':=2w'(f)=1$ by $\frac{1}{m'} \ge \ct(X,S) > \frac{1}{2}$. In particular, $z\in f$. This implies $m\leq a$ and thus $\ct(X,S)\ge 1$, a contradiction. This completes the proof of Proposition \ref{cD22}.
\end{proof}

\begin{lem} \label{1/2(2d+1,2d+1,1,2, 2d+1)}
Keep the notations as in case 2 above. Suppose that $\sigma': Y' \to X$ be the weighted blow up with weight $w'=\frac{1}{2}(2d+1,2d+1,1,2, 2d+1)$. Then the exceptional set of $\sigma'$ is a prime divisor of $Y'$.
\end{lem}

\begin{proof}
It is convenient to consider {\it truncated weight} $v=\frac{1}{2}(a,1)$ and $v'=\frac{1}{2}(1,1)$ on $\{z, u \}$.  We have  $v' \succeq \frac{1}{a} v$ and that $v'(z^ju_k)=\frac{1}{a} v(z^jv^k)$ if and only if $k=0$.

It follows immediately that for $\frak{m} \in p(z^2,u)$ in $\varphi_1$ one has
$$w'(\frak{m})=v'(\frak{m}) \ge \frac{1}{a} v(\frak{m}) \ge \frac{1}{a}(r+2) =   2d+1,$$
and for $\frak{m} \in q(z^2,u)zu $ in $\varphi_2$, one has
$$w'(\frak{m})=v'(\frak{m}) > \frac{1}{a} v(\frak{m}) \ge \frac{1}{a}\cdot \frac{r+2}{2} =  \frac{2d+1}{2}.$$

Hence $w'(\varphi_1)=2d+1$ and $w'(\varphi_2)=\frac{2d+1}{2}$.
One sees that $$\varphi_1^{w'}=x^2+yt+\lambda z^{4d+2}, \ \ \varphi_2^{w'}=z^{2d+1}+t,$$ for some constant $\lambda\in \mathbb{C}$. Note that $\lambda$ may be zero (for instance, see Example 8.2 in \cite{Kawakita05}). Therefore,
\begin{align*}
E'&=\{x^2+yt+\lambda z^{4d+2}=z^{2d+1}+t=0\}\subset\mathbb{P}(2d+1,2d+1,1,2,2d+1)\\
&\simeq \{x^2-yz^{2d+1}+\lambda z^{4d+2}=0\}\subset\mathbb{P}(2d+1,2d+1,1,2)
\end{align*}
It follows from Lemma \ref{irred}.(1) that $(x^2-yz^{2d+1}+\lambda z^{4d+2})$ is a prime ideal of the graded polynomial ring $\bC[x,y,z,u]$ with wt$(x,y,z,u)=(2d+1,2d+1,1)$. Thus, $E'$ is a prime divisor of $Y'$ and we have verified Lemma \ref{1/2(2d+1,2d+1,1,2, 2d+1)}.
\end{proof}

We need the following.

\begin{lem} \label{cD2lower2}
Keep the notations as in case 2 above. Suppose that there is another weight $w'=\frac{1}{2}(r'+2, r', a', 2, r'+4)$ satisfying $r'+2=a'(2d+1)$ and $a' \leq a$. Let $\sigma': Y' \to X$ be the weighted blow up with weight $w'$. Then the exceptional divisor of $\sigma'$ is a prime divisor.
\end{lem}

\begin{proof}
The argument is similar to that of Lemma \ref{cDlower2}.
Note that if $a'=a$ then $w=w'$ and $\sigma=\sigma'$. We may assume that $a'<a$.
As in Lemma \ref{lower}, it is convenient to consider {\it truncated weight} $v=\frac{1}{2}(a,1)$ and $v'=\frac{1}{2}(a',1)$ on $\{z, u \}$.  We have  $v' \succeq \frac{a'}{a} v$. Also, one sees $v'(z^ju^k)=\frac{a'}{a} v(z^ju^k)$ if and only if $k=0$.

It follows immediately that for $\frak{m} \in p(z^2,u)$ in $\varphi_1$ one has
$$w'(\frak{m})=v'(\frak{m}) \ge \frac{a'}{a} v(\frak{m}) \ge \frac{a'}{a}(r+2) =   r'+2,$$
and for $\frak{m} \in q(z^2,u)zu $ in $\varphi_2$, one has
$$w'(\frak{m})=v'(\frak{m}) > \frac{a'}{a} v(\frak{m}) \ge \frac{a'}{a}\cdot \frac{r+2}{2} =  \frac{r'+2}{2}.$$

Hence $w'(\varphi_1)=r'+2$ and $w'(\varphi_2)=\frac{r'+2}{2}$.
One sees that $$\varphi_1^{w'}=x^2+yt+\lambda z^{4d+2}, \ \ \varphi_2^{w'}=yu+z^{2d+1},$$  for some constant $\lambda\in \mathbb{C}$. Note that $\lambda$ may be zero (for instance, see Example 8.2 in \cite{Kawakita05}). Therefore,
\[ E'=\{x^2+yt+\lambda z^{4d+2}=yu+z^{2d+1}=0\}\subset\mathbb{P}(r'+2, r', a', 2, r'+4).\]
Let $U_y=\{y\neq 0\}, U_u:=\{u\neq 0\}$ and $U_t:=\{t\neq 0\}$ be three affine open charts of $\mathbb{P}(r'+2, r', a', 2, r'+4)$.
Note that
\[E'\cap U_y\simeq \{x^2+t+\lambda z^{4d+2}=u+z^{2d+1}=0\}\subset \hat{\mathbb{C}}_{x,z,u,t}^4/\frac{1}{r'}(2,a',2,4)\simeq \hat{\mathbb{C}}_{x,z}^2/\frac{1}{r'}(2,a'),\]
which is a prime divisor. Next, we have
\begin{align*}
E'\cap U_u&\simeq \{x^2+yt+ \lambda z^{4d+2}=y+z^{2d+1}=0\}\subset \hat{\mathbb{C}}_{x,y,z,t}^4/\frac{1}{2}(1,1,1,0)\\
&= \{x^2-z^{2d+1}t+\lambda z^{4d+2}=y+z^d=0\}\subset \hat{\mathbb{C}}_{x,y,z,t}^4/\frac{1}{2}(1,1,1,0)\\
&\simeq \{x^2-z^{2d+1}t+\lambda z^{4d+2}=0\}\subset \hat{\mathbb{C}}_{x,z,t}^3/\frac{1}{2}(1,1,0).
\end{align*}
By Lemma \ref{irred}.(1), the subset $\{x^2-z^{2d+1}t+\lambda z^{4d+2}=0\}\subset \hat{\mathbb{C}}_{x,z,t}^3$ is a prime divisor. So $E'\cap U_u$ is a prime divisor of $U_u$.
Finally, $E'\cap U_t$ is isomorphic to
\begin{align*}
&\{x^2+y+ \lambda z^{4d+2}=yu+z^{2d+1}=0\}\subset \hat{\mathbb{C}}_{x,y,z,u}^4/\frac{1}{r'+4}(r'+2,r',a',2)\\
&\simeq \{ -(x^2+\lambda z^{4d+2})u+z^{2d+1}=0\}\subset \hat{\mathbb{C}}_{x,z,u}^3/\frac{1}{r'+4}(r'+2,a',2).
\end{align*}
By Lemma \ref{irred}.(2), $\{ -(x^2+\lambda z^{4d+2})u+z^{2d+1}=0\}\subset \hat{\mathbb{C}}_{x,z,u}^3$ is a prime divisor. Thus, $E'\cap U_t$ is a prime divisor of $U_t$.

Since the three open subsets $E'\cap U_y$, $E'\cap U_u$ and $E'\cap U_t$ of $E'$ are prime divisors and $E'\cap U_y\cap U_u\cap U_t\neq \emptyset$ and $E'\cap \{y=u=t=0\}=\emptyset$, $E'$ is a prime divisor and this verifies Lemma \ref{cD2lower2}.
\end{proof}

By Lemmas \ref{cD2lower} and  \ref{cD2lower2}, we observe the following.
\begin{prop}\label{cD2ACC}
The Assumption A holds for $cD/2$. Hence
$\mathcal{T}_{3,cD/2}^{\textup{can}}$ satisfies the ascending chain condition.
\end{prop}

\begin{proof}
Suppose on the contrary that there is an infinite increasing sequence $\ct_1<\ct_2<\ct_3<\cdots$ with each $\ct_k\in \mathcal{T}_{3,cD/2}^{\textup{can}}$.
Note that each $\ct_k=\ct(X_k,S_k)$ is realized by the divisorial contraction $\sigma_k: Y_k \to X_k$.
Theorem 1.2(ii) in \cite{Kawakita05} shows that $P_k\in X_k$ and $\sigma_k$ are described in case 1 and case 2.

\noindent
{\bf Case 1.} Suppose $P_k\in X_k$ is defined by the analytical identification
\[o\in (\varphi_k: x^2+xzq_k(z^2,u)+y^2u+\lambda_k y z^{2\alpha_k -1}+p_k(z^2,u)=0) \textup{ in }\hat{\mathbb{C}}^4/\frac{1}{2}(1,1,1,0),\] where $o$ denotes the origin of $\hat{\mathbb{C}}^4/\frac{1}{2}(1,1,1,0)$
and $\sigma_k$ is a weighted blow up with weight $w_k=wt(x,y,z,u)=\frac{1}{2}(r_{k}+2,r_{k},a_k,2)$ satisfying the following:
\begin{itemize}
\item $ r_k+1=a_k d_k $ where both $a_k$ and $r_k$ are odd;
\item $w_k(\varphi_k)=w_k(y^2u)=r_k+1$; furthermore, $w_k(xzq_k(z^2,u))=r_k+1$ if $q_k\neq 0$.
\item $z^{2d_k}\in p_k(z^2,u)$ by Claim \ref{z^2dcD2}. 
\end{itemize}

Passing to a subsequence, we may assume that sequences $\{a_k\}, \{d_k\}$ are non-decreasing. By \cite[Lemma 2.5]{Stepanov}, we may assume the sequence of Newton polytopes  $\{\Gamma^+(f_k)\}$ is non-increasing.

We now pick some $i<j$ such that $a_i < a_j$. Let $s^i_j$ be the integer with $s^i_j+1=a_id_j$ and let $w^i_j:=\frac{1}{2}(s^i_j+2,s^i_j,a_i,2)$.
Clearly, $w_i \preceq w^i_j$.
Note that $f_j$ is semi-invariant under the weight $w_i$.
In particular, $w_i(f_j)$ is well-defined.
Also, since $\Gamma^+(f_i) \supseteq \Gamma^+(f_j)$, one has $w_i(f_i) \le w_i(f_j)$.
By Lemma \ref{cD2lower}, the Assumption A holds and thus the ACC follows.

\noindent {\bf Case 2.} Suppose $P_k\in X_k$ is defined by the analytical identification
$$o\in \left( \begin{array}{ll}
\varphi_{k1}:= x^2+yt+p_k(z^2,u)=0 \\
 \varphi_{k2}:= yu+z^{2d_k+1}+q_k(z^{2},u)zu+t=0 \\
\end{array} \right) \textup{ in } \hat{\mathbb{C}}_{x,y,z,u,t}^5/\frac{1}{2}(1,1,1,0,1),
$$ where $o$ denotes the origin of $\hat{\mathbb{C}}^5/\frac{1}{2}(1,1,1,0,1)$
and each $\sigma_k$ is a weighted blowup of weight $w_k=wt(x,y,z,u,t)=\frac{1}{2}(r_{k}+2,r_{k},a_k,2,r_k+4)$ satisfying the following:

\begin{itemize}
\item $r_k+2=a_k(2d_k+1)$ where $d_k$ is a positive integer;
\item $w(\varphi_{k1})=r_k+2$;
\item $w(\varphi_{k2})=\frac{r_k+2}{2}$; moreover, $w(q_k(z^2,u)zu)=\frac{r_k+2}{2}$ if $q_k(z^2,u)\neq 0$.
\end{itemize}

Passing to a subsequence, we may assume that sequences $\{a_k\}, \{d_k\}$ are non-decreasing. By \cite[Lemma 2.5]{Stepanov}, we may assume the sequence of Newton polytopes  $\{\Gamma^+(f_k)\}$ is non-increasing.

We now pick some $i<j$ such that $a_i < a_j$.
Let $s^i_j$ be the integer with $s^i_j+2=a_id_j$ and $w^i_j:=\frac{1}{2}(s^i_j+2,s^i_j,a_i,2,s^i_j+4)$.
Clearly, $w_i \preceq w^i_j$.
Note that $f_j$ is semi-invariant under the weight $w_i$ and hence $w_i(f_j)$ is well-defined.
Also, since $\Gamma^+(f_i) \supseteq \Gamma^+(f_j)$, one has $w_i(f_i) \le w_i(f_j)$.
By Lemma \ref{cD2lower2}, the Assumption A holds and thus the ACC follows.
\end{proof}

\section{Summary and open problems}

Recall that the set $\mathcal{T}_{3}^{\textup{can}}$ of threefold canonical thresholds has a decomposition in $\dagger_2$:
 $$\mathcal{T}_{3}^{\textup{can}} = \aleph_4 \cup \mathcal{T}_{3,sm}^{\textup{can}} \cup \mathcal{T}_{3, cA}^{\textup{can}} \cup \mathcal{T}_{3, cA/n}^{\textup{can}} \cup \mathcal{T}_{3,cD}^{\textup{can}} \cup \mathcal{T}_{3, cD/2}^{\textup{can}},$$
where $$\aleph_4 =\{\ \frac{q}{n}\ |\ q \textup{ and $n$ are positive integers with }q\le 4\}$$ is an ACC set.
By \cite[Theorem 1.7]{Stepanov}, Propositions \ref{cAACC}, \ref{cAnACC}, \ref{cDACC}, \ref{cD2ACC} and Theorem \ref{3-folddc}, we obtain Theorem \ref{3CTACC}.
Also, it is easy to verify that
$$ \aleph_4\cap (\frac{1}{2},1)\subseteq \{\frac{1}{2}+\frac{1}{n}\}_{n\ge 3} \cup \{\frac{4}{5}\}.$$
It follows from \cite[Example 3.11]{Prok08} that $\ct_P(X,S)=\frac{4}{5}\in \aleph_4\cap (\frac{1}{2},1)$ when $P\in X$ is of type $cA_1$ given by $xy+x^7+z^2+u^3=0$ and $S$ is the Weil divisor $y=0$.
By Propositions \ref{sm2}, \ref{cA2}, \ref{cAn2}, \ref{cD2}, \ref{cD22} and Theorem \ref{3-folddc},
we derive Theorem \ref{3CT2}.

From our studies, it seems natural to consider the following problems.

\begin{ques}
For every integer $k>2$, is it possible to determine completely the set $\mathcal{T}_{3}^{\textup{can}}\cap (\frac{1}{k},\frac{1}{k-1})$ ?
\end{ques}

Recall that canonical threshold is involved in Sarkisov degree $(\mu, c, e)$, which is used to prove the termination of Sarkisov program. It is natural to ask the following.
\begin{ques}
Can we have an effective bound on ``length of Sarkisov program'' in dimension three?
\end{ques}


\begin{thebibliography}{99}


\bibitem[Chen14]
{CJK14}J. A. Chen, {\em Factoring threefold divisorial contractions to points,}
\textit{Annali della Scuola Normale Superiore di Pisa. Classe di Scienze. Serie V} $\bf{13}$ (2014), no. 2, 435-463.

\bibitem[Chen15]
{CJK15}J. A. Chen,  {\em Birational maps of 3-folds}, \textit{Taiwanese J. Math.} $\bf{19}$ (2015), no. 6, 1619-1642.


\bibitem[CH11]
{CH11} J. A. Chen and C. D. Hacon, {\em Factoring $3$-fold flips and divisorial contractions to curves.} \textit{J. Reine Angew Math.} {\bf 657} (2011), 173-197.

\bibitem[Chen22]{Ch22} J-J. Chen,  {\em Accumulation points on 3-fold canonical thresholds}, arXiv: 2202.06239.

\bibitem[Cor95]{Corti} A. Corti, {\em Factoring birational maps of 3-folds after Sarkisov}, \textit{J. Algebraic Geom.} $\bf{4}$ (1995), 223-254.




\bibitem[H99]
{Haya99}
T. Hayakawa, {\em Blowing ups of 3-dimensional terminal singularities}, \textit{Publ. Res. Inst. Math. Sci.} {\bf 35} (1999), no. 3, 515-570.

\bibitem[H00]{Haya00} T. Hayakawa, {\em Blowing ups of 3-dimensional terminal singularities II}, \textit{Publ. Res. Inst. Math. Sci.} {\bf 36} (2000), no. 3, 423-456.

\bibitem[HMX14]
{HMX14}
C. D. Hacon, J. Mckernan and C. Y. Xu, {\em ACC for log canonical thresholds}, \textit{Ann. of Math.}  (2) {\bf180} (2014), no. 2, 523-571.


\bibitem[HLL22]{HLL22} J. Han, J. Liu and Y. Luo, {\em ACC for minimal log discrepancies of terminal threefolds}, arXiv: 2202.05287v2.


\bibitem[Hos16]{Hosgood16} T. Hosgood, {\em An introduction to varieties in weighted projective space}. arXiv: 1604.02441v5.




\bibitem[Kwk01]{Kawakita01} M. Kawakita, {\em Divisorial contractions in dimension three which contract divisors to smooth points}, \textit{Invent. Math.} $\bf{145}$ (2001), no. 1, 105-119.

\bibitem[Kwk02]{Kawakita02} M. Kawakita, {\em Divisorial contractions in dimension three which contract divisors to compound A1 points}, \textit{Compos. Math.} $\bf{133}$ (2002), no. 1, 95-116.

\bibitem[Kwk05]{Kawakita05}
M. Kawakita, {\em Threefold divisorial contractions to singularities of higher indices}, \textit{Duke Math. J.} {\bf 130} (2005), no. 1, 57-126.

\bibitem[Kaw88]{Kaw88}
Y. Kawamata, Crepant blowing-up of 3-dimensional canonical singularities and its application to degenerations of surfaces. \textit{Ann. Math.} (2) {\bf 127}, no. 1, (1988), 93-163.


\bibitem[Kaw96]
{Ka} Y. Kawamata, {\em Divisorial contractions to $3$-dimensional terminal quotient singularities}, Higher-dimensional complex varieties (Trento, 1994), 241-246, de Gruyter, Berlin 1996.

\bibitem[KMM87]
{KMM87} T. Kawamata, K. Matsuda and K. Matsuki, {\em Introduction to the minimal model problem}, Algebraic geometry, Sendai, 1985, 283–360, Adv. Stud. Pure Math., {\bf 10}, North-Holland, Amsterdam, 1987.


\bibitem[Khov]
{Khov}
A. G. Khovanski$\breve{\textup{i}}$,
{\em Sums of finite sets, orbits of commutative semigroups and Hilbert functions}, (Russian. Russian summary)
\textit{Funktsional. Anal. i Prilozhen.} $\bf{29}$ (1995), no. 2, 36-50; translation in \textit{Funct. Anal. Appl.} $\bf{29}$ (1995), no. 2, 102-112.



\bibitem[Kol92]{Kol92} J. Koll\'ar, et al., {\em Flips and abundance for algebraic threefolds:
Papers from the Second Summer Seminar on Algebraic Geometry held at the University of Utah, Salt Lake City, Utah, August 1991. Astérisque}, no. 211. Société Mathématique de France, Paris, 1992.


\bibitem[Kol97]{Kol97} J. Koll\'ar, {\em Singularities of pairs}, \textit{Algebraic geometry-Santa Cruz 1995, Proc. Sympos. Pure Math.},  vol. 62, Amer. Math. Soc., Providence, RI, 1997. pp. 221-287.


\bibitem[KM92]{KM92}
J. Koll\'{a}r and S. Mori, {\em Classification of three-dimensional flips},
\textit{J. Amer. Math. Soc.} {\bf 5} (1992), no. 3, 533-703.


\bibitem[Mat02]{Mat02} K. Matsuki, {\em Introduction to the Mori program}, Universitext. Springer-Verlag, New York, 2002.

\bibitem[MP04]{MP04} J. $\textup{M}^c$Kernan and Y. Prokhorov,  {\em Threefold thresholds}, \textit{Manuscripta Math.}, {\bf114} (2004), no. 3, 281-304.


\bibitem[Mori85]{Mori85} S. Mori, {\em 3-dimensional terminal singularities}, \textit{Nagoya Math. J.} {\bf 98} (1985), 43-66.







\bibitem[Prok08]{Prok08}
Y. Prokhorov, {\em Gap conjecture for 3-dimensional canonical thresholds},
\textit{J. Math. Sci. Univ. Tokyo} {\bf15} (2008), no. 4, 449-459.

\bibitem[Sho85]{Sho85}
V. V. Shokurov, {\em A nonvanishing theorem},
\textit{Izv. Akad. Nauk SSSR Ser. Mat.} {\bf49} (1985), no. 3, 635–651.


\bibitem[Stepa11]
{Stepanov} D. A. Stepanov, {\em Smooth three-dimensional canonical thresholds}, (Russian) \textit{Mat. Zametki} {\bf 90} (2011), no. 2, 285-299; translation in \textit{Math. Notes} {\bf90} (2011), no. 1-2, 265-278.



\bibitem[YPG]{YPG} M. Reid, Young person's guide to canonical singularities.
\textit{Proc. Sympos. Pure Math.} {\bf 46} (1987), 345-414.

\end{thebibliography}
\end{document}